\documentclass[12pt]{article}

\usepackage[dvipsnames]{xcolor}
\usepackage[all]{xy}
\usepackage{amssymb}
\usepackage{amsmath}
\usepackage{amsthm}
\usepackage{amscd}
\usepackage{amsfonts}
\usepackage{mathrsfs}
\usepackage{mathtools}
\usepackage{mathabx}
\usepackage{enumerate}
\usepackage{times}
\usepackage[pdftex]{hyperref}
\usepackage[most]{tcolorbox}
\usepackage{titletoc}
\usepackage{moresize}
\usepackage{anyfontsize}
\usepackage[indentafter]{titlesec}
\usepackage{cleveref}
\usepackage[utf8]{inputenc}
\usepackage[T1]{fontenc}
\usepackage{lmodern}
\usepackage{comment}
\usepackage{authblk}
\usepackage{multicol}
\usetikzlibrary{arrows}
\usepackage{tikz-cd}
\usetikzlibrary{arrows}

\theoremstyle{plain}
\newtheorem{theo}{Theorem}[section]
\newtheorem{prop}[theo]{Proposition}
\newtheorem{lemma}[theo]{Lemma}
\newtheorem{coro}[theo]{Corollary}

\newtheorem*{corollary}{Corollary}

\newtheorem{theol}{Theorem}

\theoremstyle{definition}
\newtheorem{df}[theo]{Definition}

\newtheorem{ex}[theo]{Example}

\makeatletter
\newtheorem*{rep@theorem}{\rep@title}
\newcommand{\newreptheorem}[2]{%
\newenvironment{rep#1}[1]{%
 \def\rep@title{#2 \ref{##1}}%
 \begin{rep@theorem}\it}%
 {\end{rep@theorem}}}
\makeatother
\newreptheorem{theorem}{Theorem}

\newcommand{\Z}{\mathbb{Z}}
\newcommand{\R}{\mathbb{R}}
\newcommand{\C}{\mathbb{C}}
\renewcommand{\H}{\mathbb{H}}

\newcommand{\calC}{\mathcal{C}}

\newcommand{\calG}{\mathcal{G}}
\newcommand{\calH}{\mathcal{H}}

\newcommand{\calP}{\mathcal{P}}
\newcommand{\calS}{\mathcal{S}}
\newcommand{\calU}{\mathcal{U}}
\newcommand{\cl}[1]{\overline{#1}}
\newcommand{\cic}[1]{\left\langle #1 \right\rangle}
\newcommand{\til}[1]{\tilde{#1}}

\newcommand{\abs}[1]{{#1}^{\text{abs}}}
\newcommand{\bl}{\backslash}
\newcommand{\bfG}{\mathbf{G}}
\newcommand{\bfH}{\mathbf{H}}

\providecommand{\keywords}[1]
{\small	\textbf{Keywords:} #1.}

\providecommand{\MSC}[1]
{\small	\textbf{MSC Classification:} #1.}

\DeclareMathOperator{\HNN}{HNN}
\DeclareMathOperator{\SL}{SL}
\DeclareMathOperator{\SO}{SO}
\DeclareMathOperator{\GL}{GL}
\DeclareMathOperator{\tr}{tr}
\DeclareMathOperator{\Stab}{Stab}

\everymath{\displaystyle}

\definecolor{mygreen}{rgb}{0, 0.3, 0}
\definecolor{myred}{rgb}{0.4,0,0}
\definecolor{myblue}{rgb}{0.1,0.1,0.6}

\hypersetup{
     colorlinks=true,
     linkcolor=mygreen,
     filecolor=myblue,
     citecolor = myred,      
     urlcolor=myblue,
}

\title{Prosoluble subgroups of the profinite completion of the fundamental group of compact $3$-manifolds}

\author[1]{Lucas C. Lopes\thanks{The first author was supported by CAPES.}}

\author[2]{\href{https://mat.unb.br/\~pz/index.html}{Pavel A. Zalesskii}\thanks{The second author was supported by CNPq.}}

\affil[1]{Universidade de Brasília, Departamento de Matemática, 70910-900 Brasília - DF, Brazil, email address: \href{mailto:l.c.lopes@mat.unb.br}{l.c.lopes@mat.unb.br}}

\affil[2]{Universidade de Brasília, Departamento de Matemática, 70910-900 Brasília - DF, Brazil, email address: \href{mailto:pz@mat.unb.br}{pz@mat.unb.br}}

\begin{document}

\maketitle

\begin{abstract}
We give a description of finitely generated prosoluble subgroups of the profinite completion of $3$-manifold groups and toral relatively hyperbolic virtually compact special groups.
\end{abstract}

\MSC{20E08, 20E18}

\keywords{manifolds, fundamental groups, profinite groups, acylindrical actions}

\section{Introduction}

In recent years there has been a great deal of interest in detecting properties of the fundamental group $\pi_1(M)$ of a $3$-manifold  and other groups of geometric nature via its finite quotients, or more conceptually by its profinite completion (see \cite{BRW17}, \cite{BMRS20}, \cite{CW22}, \cite{Wil18}). This motivates the study of the profinite completion $\widehat{\pi_1(M)}$ of the fundamental group of a $3$-manifold. The study of $\widehat{\pi_1(M)}$ is based mainly on the work of Henry Wilton and the second author (see \cite{WZ17}) where it is proved that the standard splittings of $3$-manifold groups as a free product with amalgamation, HNN-extension, and generally as a graph of groups (we refer to them as free constructions later) survive after profinite completion of $\pi_1(M)$. Namely $\widehat{\pi_1(M)}$ admits the (profinite) Kneser–Milnor decomposition and then the (profinite) JSJ-decomposition, and for hyperbolic manifolds with cusps (after passing to a finite sheeted cover) there is a suitable hierarchy  in which the corresponding actions on profinite trees are profinitely acylindrical.

At this point one can use the profinite analogue of Bass-Serre theory for groups acting on trees. Note, however, that this theory does not have the full strength of its classical original. The main theorem of Bass-Serre theory (that describes the subgroups of free constructions) asserts that a group $G$ acting on a tree $T$ is the fundamental group of a  graph of  groups $(\calG, G \bl T, D)$, where $D$ is a maximal subtree of $G \bl T$ and $\calG$ consists of edge and vertex stabilizers of a connected transversal of $G \bl T$. This does not hold in the profinite case. As a consequence, subgroup theorems do not hold, in general, for profinite groups and even for free profinite products, i.e., the profinite version of the Kurosh subgroup theorem is not valid. This motivates studying subgroup structure of free constructions for important subclasses of profinite groups. The most important subclasses of profinite groups are prosoluble and pro-$p$ as they play the same role as finite  soluble and finite $p$-groups  in finite group theory.
 
Such a study for pro-$p$ subgroups was performed in \cite{WZ18}, where the structure of finitely generated pro-$p$ subgroups of profinite completions of the fundamental groups of $3$-manifolds was described. The objective of this paper is to study prosoluble subgroups of $\widehat{\pi_1(M)}$.
 
We begin with the fundamental group of a compact hyperbolic $3$-manifold.

\begin{theol}\label{projective}
Let $\pi_1(M)$ be the fundamental group of a closed hyperbolic $3$-manifold $M$ and $H$ a finitely generated prosoluble subgroup of the profinite completion of $\pi_1(M)$. Then $H$ is projective, i.e., of cohomological dimension $1$.
\end{theol}

Note that the Schreier theorem does not hold for profinite groups, i.e., subgroups of free profinite groups, in general, are not free but projective. So Theorem \ref{projective} tells that prosoluble subgroups of $\widehat{\pi_1(M)}$ are isomorphic to subgroups of a free profinite group.

The proof of Theorem \ref{projective} uses the dramatic  developments of Agol (\cite{Ago13}), Kahn--Markovic (\cite{JV12}) and Wise (\cite{Wis11}), whose works implies that the fundamental groups of closed hyperbolic $3$-manifolds are also fundamental groups of compact, virtually special cube complexes (see \cite{AFH20} for a summary of these developments); such groups are called virtually compact special. Virtually special groups originated in the work of Wise (\cite{HW08}, \cite{HW12}). These groups have a central role in modern geometric group theory and Haglund and Wise showed that many hyperbolic groups are virtually special such as hyperbolic Coxeter groups and one relator groups with torsion. 

The next theorem  describes torsion-free finitely generated prosoluble subgroups of hyperbolic virtually special groups.

\begin{theol}\label{virs}
Let $G$ be a torsion-free, hyperbolic, virtually special group. Any finitely generated prosoluble subgroup $H$ of the profinite completion of $G$ is projective.
\end{theol}

Using that the standard cocompact arithmetic subgroups of $\SO(n,1)$ are virtually compact special (\cite[Theorem 1.10]{BHW11}) we deduce the following

\begin{corollary}
Prosoluble subgroups of the profinite completion of standard cocompact arithmetic subgroups of $\SO(n,1)$ are projective.
\end{corollary} 

The next objective of this paper is to extend these results to relatively hyperbolic virtually compact special groups, in particular, to the fundamental hyperbolic $3$-manifolds with cusps. In order to prove that, our arguments use both an analogue of Wise’s Quasiconvex Hierarchy Theorem proved by Eduardo Einstein \cite{Ein19} and the fact that it is preserved by the profinite completion established in \cite{Zal22new}.

\begin{theol}\label{toral hyperbolic}
Let $G$ be a torsion-free  virtually compact special toral relatively hyperbolic group and $H$ a finitely generated prosoluble subgroup of the profinite completion $\widehat{G}$ of $G$. Then one of the following statements holds:
\begin{enumerate}[(i)]
\item $H$ is virtually a subgroup of a free prosoluble product of  abelian pro-$p$ groups;
\item $H$ is a virtually abelian group.
\end{enumerate}
\end{theol}

Using the \cite[Theorem 4.3]{WZ18} for $3$-manifolds with cusps  we deduce the following:

\begin{theol}\label{sam}
Let $M$ be a  hyperbolic $3$-manifold with cusps and $H$ a finitely generated prosoluble subgroup of the profinite completion of $\pi_1(M)$. Then one of the following statements holds:
\begin{enumerate}[(i)]
\item $H$ is virtually a subgroup of a free prosoluble product of  abelian pro-$p$ groups;

\item $H$ is a virtually abelian.
\end{enumerate}
\end{theol}

Finally, analyzing $\widehat{\pi_1(M)}$ for Thurston's geometries case-by-case, we  obtain  the classification of finitely generated prosoluble subgroups of $\widehat{\pi_1(M)}$ for an arbitrary compact, orientable $3$-manifold.

\begin{theol}\label{ori}
Let $M$ be a compact orientable $3$-manifold. If $H$ is a finitely generated prosoluble subgroup of $\widehat{\pi_1(M)}$, then one of the following statements holds:
\begin{enumerate}[(i)]
\item $H$ is isomorphic to a subgroup of a free prosoluble product of pro-$p$ groups from the following list of isomorphism types:
\begin{enumerate}[(1)]
\item For $p > 3$: $C_p$; $\Z_p$; $\Z_p \times \Z_p$; the pro-$p$ completion of $(\Z \times \Z) \rtimes \Z$ and the pro-$p$ completion of a residually-$p$ fundamental group of a non-compact Seifert fibred manifold with hyperbolic base of orbifold;

\item For $p = 3$: in addition to the list of (1) we have a torsion-free extension of $\Z_3 \times \Z_3 \times \Z_3$ by $C_3$;

\item For $p = 2$: in addition to the list of (1) we have $C_{2^m}$; $D_{2^k}$; $Q_{2^n}$; $\Z_2 \rtimes C_2$; the torsion-free extensions of $\Z_2 \times \Z_2 \times \Z_2$ by one of $C_2$, $C_4$, $C_8$, $D_2$, $D_4$, $D_8$, $Q_{16}$; the pro-$2$ completion of the Klein-bottle group $\Z \rtimes \Z$; the pro-$2$ completion of all torsion-free extensions of a soluble group $(\Z \rtimes \Z) \rtimes \Z$ with a group of order at most $2$;
\end{enumerate}

\item $H \simeq \widehat{\Z}_{\pi} \rtimes C$ is a profinite Frobenius group where $C$ is a finite cyclic group and $\pi$ is set of primes;

\item $H$ is a subgroup of a central extension of $D_{2n}$, a tetrahedral group $T$, a octahedral group $O$ by a cyclic group of even order;

\item $H$ is a subgroup of the profinite completion of a $3$-dimensional Bieberbach group $B$ (i.e, $B$ is torsion-free virtually $\Z^3$);

\item $H$ is a subgroup of the profinite completion of a group extension of $\Z^2 \rtimes_A \Z$ by $C$ where $C$ is trivial or $C_2$ and $A \in \GL_2(\Z)$ (and $A$ is an Anosov matrix if $C$ is trivial), or a central extension of $\Z$ by $\Z^2$;

\item $H$ is an extension of a torsion-free procyclic group $\widehat{\Z}_{\sigma}$ by a subgroup $H_0$ of a finite free prosoluble product of finite cyclic $p$-groups with $H_0$ acting either trivially on $\widehat{\Z}_{\sigma}$  or by inversion.
\end{enumerate}
\end{theol}

In this theorem we are denoting by $T$ the tetrahedral group of $12$ rotational symmetries of a tetrahedron, by $O$ the octahedral group of $24$ rotational symmetries of a cube or of an octahedron. An Anosov matrix is a matrix $A \in \GL_2(\Z)$ determined by an Anosov homeomorphism of the torus satisfying the following conditions $\tr(A) > 2$ if $\det(A) = 1$ and $\tr(A) \neq 0$ if $\det(A) = -1$ (see \cite[Section 2]{GM23}). We use the notation $\widehat{\Z}_{\alpha} = \prod_{p \in \alpha}\Z_p$ where $\alpha$ is a set of primes. Also, recall that a profinite Frobenius group $\widehat{\Z}_{\pi} \rtimes C$ is a group such that $C$ is cyclic, $|C|$  divides $p-1$ for any $p\in \pi$ and $[c,z]\neq 1$ for all $1\neq c\in C, 0\neq z\in \widehat{\Z}_\pi$.

To prove Theorem \ref{ori} we make heavy use of the profinite version of Bass-Serre theory. In particular, we prove the following general theorem of independent interest:.

\begin{theol}\label{general}
Let $(\calG,\Gamma)$ be an injective $k$-acylindrical finite graph of profinite groups, $G$ its fundamental group. If $H$ is a finitely generated prosoluble subgroup of $G$, then
$H$ embeds into a  free profinite product $\coprod_{v \in \Gamma} H_v$, where $H_v = H \cap \calG(v)^g$ and $g \in G$. Moreover, if $H\neq H_v$ for some $v$, then either $H_v$ is finitely generated pro-$p$ or $H \simeq \widehat{\Z}_{\pi} \rtimes C$ is a profinite Frobenius group.
\end{theol}

The celebrated Sela theorem (see \cite[Theorem 3.2]{Sel97}) states that a finitely generated group $G$ is $k$-acylindrically accessible, i.e., if $G$ is the fundamental group of a reduced $k$-acylindrical  graph of groups $(\calG, \Gamma)$, then $\Gamma$ is finite and, in fact, bounded in terms of $k$ and the minimal number of generators $d(G)$. This is equivalent to the statement that  $G$ acts on the corresponding Bass-Serre tree with  bounded number of maximal  vertex stabilizers up to conjugation. These statements are not equivalent, however, for profinite groups (one can see it using Proposition  \ref{prop1}). We suggest that the correct profinite analog of accessibility is a bound for the number of maximal vertex stabilizers up to conjugation. Using Theorem \ref{general} we prove an analogue of Sela's theorem for prosoluble groups.

\begin{theol}\label{sela}
Let $(\calG,\Gamma)$ be an injective $k$-acylindrical finite graph of profinite groups and $G$ its profinite fundamental group. If $H$ is a finitely generated prosoluble subgroup of $G$, then  the number of maximal vertex stabilizers of $H$ acting on the standard profinite Bass-Serre tree, up to conjugation, is at most $d(H)$.
\end{theol}

This paper is organized as follows. In Section \ref{sec2} we introduce the basic ideas and notations of the profinite Bass-Serre theory as well as prove auxiliary results. In Section \ref{sec3}, we study pro-$\calC$ groups acting $k$-acylindrically on a profinite simply connected graph. Closed pro-$\calC$ subgroups of the profinite fundamental groups of a finite graph of profinite groups  are studied in Section \ref{sec4}, where we prove Theorem \ref{sela}. We prove Theorem \ref{general} in  Section \ref{sec5}. Section \ref{sec6} is dedicated to prosoluble subgroups of virtually special groups, where Theorems \ref{projective}, \ref{virs}, \ref{toral hyperbolic} and \ref{sam} are proved. We aplly them  in Section \ref{sec7} to the profinite completion of arithmetic hyperbolic manifold groups. In Section \ref{sec8} we deal with Thurston's geometries of a compact $3$-manifold in order to prove Theorem \ref{ori}.

When we deal with profinite groups our subgroups are assumed to be closed, homomorphisms to be continuous, finitely generated means topologically finitely generated. Throughout the paper we denote by $\calC$ a class of finite groups closed for subgroups, quotients and extensions. We shall restate the results of the introduction inside of the section for easier reading.

\section{Finite graphs of pro-$\calC$ groups and its fundamental group}\label{sec2}

In this section we will introduce the basic definitions and notations for profinite Bass-Serre theory used in this paper. This section is entirely based on \cite{Rib}.

\subsection{Finite graphs of pro-$\calC$ groups}

\begin{df}[Profinite graphs]
A {\it profinite graph} is a profinite space $\Gamma$ with a distinguished non-empty subset $V(\Gamma)$ and two continuous maps $d_i: \Gamma \to V(\Gamma)$, $i = 0,1$, which are the identity on $V(\Gamma)$.
\end{df}

The subset $V(\Gamma)$ is the {\it vertex-set} of $\Gamma$ and $E(\Gamma) = \Gamma - V$ is the {\it edge-set} of $\Gamma$. If $e \in E(\Gamma)$, then we have
$$\begin{tikzpicture}[-latex, auto, node distance=3cm,
semithick]
\node[circle,draw=black, fill=black, inner sep=0pt, minimum size=1.5mm] at (0,0) {};
\node[circle,draw=black, fill=black, inner sep=0pt, minimum size=1.5mm] at (2,0) {};
\draw node[above] {$d_0(e)$} (0,0) -- node[below] {$e$} (2,0) node[above] {$d_1(e)$};
\end{tikzpicture}$$
We call $d_0(e)$ and $d_1(e)$, respectively, by {\it initial} and {\it terminal} vertex. The maps $d_0,d_1$ are the {\it incidence maps}.

\begin{df}
A {\it morphism} $\varphi: \Gamma \to \Delta$ of profinite graphs is a continuous map such that $\varphi d_i = d_i \varphi$ for $i = 0,1$.
\end{df}

\begin{df}
A profinite graph $\Gamma$ is {\it connected} if every finite quotient graph of $\Gamma$ is  connected as an abstract graph.
\end{df}

\begin{df}[Simply connected profinite graphs]\label{simply connected}
Let $G$ be a pro-$\calC$ group acting freely on a profinite graph $\Gamma$. The natural epimorphism $\zeta: \Gamma \to G \bl \Gamma$ is called a {\it Galois $\calC$-covering} of $\Delta = G \bl \Gamma$ and the associated group is denoted by $G = G(\zeta)$. If $\Gamma$ does not have non-trivial Galois $\calC$-coverings, then $\zeta$ is called {\it universal $\calC$-covering} and, in this case, $\pi_1(\Gamma) = G(\Gamma | \Delta)$ is the {\it pro-$\calC$ fundamental group} of $\Gamma$. If $\pi_1(\Gamma) = 1$ then we say that $\Gamma$ is {\it $\calC$-simply connected}.
\end{df}

Let $\Gamma$ be a profinite graph. Define $E^{\ast}(\Gamma) = \Gamma \bl V(\Gamma)$ and $\widehat{\Z}_{\widehat{\calC}}$ as the pro-$\calC$ completion of $\Z$. 

\begin{df}[Profinite trees]
We say that $\Gamma$ is a {\it pro-$\calC$ tree} if the sequence
$$0 \buildrel{}\over\longrightarrow [[\widehat{\Z}_{\widehat{\calC}}(E^{\ast}(\Gamma),\ast)]] \buildrel{D}\over\longrightarrow [[\widehat{\Z}_{\widehat{\calC}}(V(\Gamma))]] \buildrel{\varepsilon}\over\longrightarrow \widehat{\Z}_{\widehat{\calC}} \buildrel{}\over\longrightarrow 0$$
is exact, where $d(\bar{e}) = d_1(e) - d_0(e)$ for the image $\bar{e}$ of $e \in E(\Gamma)$ in $E^{\ast}(\Gamma)$ with $d(\ast) = 0$ and $\varepsilon(v) = 1$ for $v \in V(\Gamma)$.
\end{df}

In \cite[Section 3.10]{Rib} we can see that $``$every simply connected profinite graph is a profinite tree$"$, $``$if $\calC$ is the variety of all finite soluble groups, $\calC$-simply connected and $\calC$-tree are equivalent$"$.

We say that  a profinite group $G$
acts on the profinite tree $T$ if it acts continuously on
$T$ and the action commutes with $d_0$ and $d_1$. We
 denote by $G_t$ the stabilizer of $t \in T$ in $G$.

We define
$$T^g = \{m \in T: gm = m\}$$
to be the set of $g$-fixed points for a pro-$\calC$ tree $T$, then $T^g$ is again a $\calC$-tree (see \cite[Theorem 4.1.5]{Rib}). The subgroup of $G$ generated by all vertex stabilizers usually will be denoted by $\widetilde G$.

\begin{df}
Let $\Gamma$ be a pro-$\calC$ tree. Given the vertices $v,w \in \Gamma$ we can consider $[v,w]$ as the intersection of all pro-$\calC$ subtrees of $\Gamma$ containing $v$ and $w$. We call $[v,w]$ a {\it geodesic}. The {\it length $\ell([v,w])$} of $[v,w]$ is the number of edges in $[v,w]$.
\end{df}

\begin{df}[Finite graphs of profinite groups]
A {\it finite graph of pro-$\calC$ groups} is a pair $(\calG,\Gamma)$ where $\Gamma$ is a connected finite graph, $\calG = \dot{\bigcup_{m \in \Gamma}}\calG(m)$ with $\calG(m)$ a pro-$\calC$ group for each $m$ and there are continuous monomorphisms $\partial_i: \calG(e) \to \calG(d_i(e))$ for each edge $e \in E(\Gamma)$.
\end{df}

\begin{df}
A {\it morphism}
$$\underline{\alpha} = (\alpha,\alpha'): (\calG,\Gamma) \to (\calG',\Gamma')$$ of
graphs of pro-$\calC$ groups consists of a pair of continuous maps $\alpha: \calG \to \calG'$ and $\alpha': \Gamma \to \Gamma'$ such that the diagram
$$\begin{tikzpicture}[node distance=3cm, auto]
\node(calG) at (0,0) {$\calG$};
\node(calG') at (3,0) {$\calG'$};
\node(T) at (0,-3) {$\Gamma$};
\node(T') at (3,-3) {$\Gamma'$};
\path[->] (calG) edge [] node[above] {$\alpha$} (calG');
\path[->] (T) edge [] node[below] {$\alpha'$} (T');
\path[->] (calG) edge [] node[left] {$\pi$} (T);
\path[->] (calG') edge [] node[right] {$\pi'$} (T');
\end{tikzpicture}$$
is commutative. We say that $\underline{\alpha}$ is a {\it monomorphism} if $\alpha$ and $\alpha'$ are injective.
\end{df}

\begin{df}
A finite graph of pro-$\calC$ groups $(\calG,\Gamma)$ is {\it reduced} if for every edge $e$ which is not a loop (that is, $d_0(e) \neq d_1(e)$), neither $\partial_0: \calG(e) \to \calG(d_0(e))$ nor $\partial_1: \calG(e) \to \calG(d_1(e))$ is an isomorphism. We say that an edge $e$ is fictitious if $e$ is not a loop and $\partial_0$ or $\partial_1$ is an isomorphism.
\end{df}

Given a finite graph of groups with fictitious edges $e_1,...,e_n$ we can collapse this edges transforming it in a reduced finite graph of groups with the same fundamental group (see \cite[Remark 2.8]{CZ}).

\subsection{The fundamental group of a finite graph of pro-$\calC$ groups}

For a totally new construction of the fundamental group of a profinite graph of pro-$\calC$ groups using an approach similar to the abstract case, the reader should check \cite{AZ}. Here we will use the standard construction done in \cite{Rib}.

\begin{df}[The fundamental group of a finite graph of profinite groups]
Let $(\calG,\Gamma)$ be a finite graph of pro-$\calC$ groups and $T$ be a maximal subtree of $\Gamma$. {\it The pro-$\calC$ fundamental group} of $(\calG,\Gamma)$ is a group $G$ with continuous homomorphisms
$$\nu_m: \calG(m) \to G$$
and continuous maps $E(\Gamma) \to G$ denoted by $e \mapsto t_e$ such that $t_e = 1$ if $e \in E(T)$ and
$$(\nu_{d_0(e)}\partial_0)(x) = t_e(\nu_{d_1(e)}\partial_1)(x)t_e^{-1}$$
for all $x \in \calG(e)$ and $e \in E(\Gamma)$ satisfying the following universal property: if $H$ is a pro-$\calC$ group, $\beta_m: \calG(m) \to H$ are continuous homomorphisms, $e \to s_e$ ($e \in E(\Gamma)$) is a map with $s_e = 1$ if $e \in E(T)$ and
$$(\beta_{d_0(e)}\partial_0)(x) = t_e(\beta_{d_1(e)}\partial_1)(x)t_e^{-1}$$
for all $x \in \calG(e)$ and $e \in E(\Gamma)$, then there is a unique continuous homomorphism $\delta: G \to H$ such that $\delta(t_e) = s_e$ and $\delta \nu_m = \beta_m$ for each $m \in \Gamma$.
\end{df}

We will denote the fundamental group of $(\calG,\Gamma)$ by
$$G = \Pi_1(\calG,\Gamma).$$

The reader can find proofs of existence and uniqueness of fundamental groups in \cite{Rib}.

The basic free constructions such as free profinite groups, free profinite products, free profinite products with amalgamation and $\HNN$-extensions are the most basic examples of fundamental groups of finite graphs of groups (see \cite[Section 6.2]{Rib}).

The following lemma will be useful:

\begin{lemma}\label{aux}
Let $(\calG,\Gamma)$ be an injective profinite graph of pro-$\calC$ groups and $\Delta$ a connected subgraph of $\Gamma$. Then $\Pi_1(\calG,\Delta)$ is embedded in $\Pi_1(\calG,\Gamma)$.
\end{lemma}

\begin{proof}
First assume that $(\calG,\Gamma)$ is a finite graph of finite groups. Note that $\Pi_1(\calG,\Gamma)$ and $\Pi_1(\calG,\Delta)$ are pro-$\calC$ completions of the abstract fundamental groups $\abs{\Pi_1}(\calG,\Gamma)$ and $\abs{\Pi_1}(\calG,\Delta)$, respectively. Consider the collapse $\Gamma \to \Gamma/\Delta$ and the graph of groups $(\calG_{\Delta}, \Gamma/\Delta)$ defining $\calG_{\Delta}(m) = \abs{\Pi_1}(\calG,\Delta)$ if $m$ is the image of $\Delta$ and $\calG_{\Delta}(m) = \calG(m)$ for the others $m$. Since the edge groups of $(\calG_{\Delta}, \Gamma/\Delta)$ are finite, the pro-$\calC$ topology of $\abs{\Pi_1}(\calG_{\Delta},\Gamma/\Delta) = \abs{\Pi_1}(\calG,\Gamma)$ induces the full pro-$\calC$ topology on $\abs{\Pi_1}(\calG,\Delta)$ (see \cite[Proposition 6.5.5]{Rib}). Then by \cite[Proposition 6.5.3]{Rib} the graph of groups $(\overline{\calG}_{\Delta},\Gamma/\Delta)$ of the pro-$\calC$ completions of vertex and edge groups is injective. It follows that the pro-$\calC$ completion of $\abs{\Pi_1}(\calG,\Delta)$ embeds in the pro-$\calC$ completion of $\abs{\Pi_1}(\calG,\Gamma)$, as desired.

For the general case, given any open normal subgroup $U$ of $G$, consider
$$\til{U} = \cic{U \cap \calG(v)^g : v \in V(\Gamma), g \in G}.$$
We have (see \cite[Proposition 2.15 and Theorem 3.9]{AZ})
$$G_U = G/\til{U} = \Pi_1(\calG_U,\Gamma)$$
such that
$$\Pi_1(\calG,\Gamma) = \varprojlim \Pi_1(\calG_U,\Gamma)$$
and
$$\Pi_1(\calG,\Delta) = \varprojlim \Pi_1(\calG_U,\Delta).$$
Since the diagram
$$\begin{tikzpicture}[auto]
\node(A) at (0,0) {$\Pi_1(\calG,\Gamma)$};
\node(B) at (3,0) {$\Pi_1(\calG,\Delta)$};
\node(C) at (0,-2) {$\Pi_1(\calG_U,\Gamma)$};
\node(D) at (3,-2) {$\Pi_1(\calG_U,\Delta)$};
\path[->] (A) edge[] node[above] {} (B);
\path[right hook->] (C) edge[] node[below] {} (D);
\path[->] (A) edge[] node[left] {} (C);
\path[->] (B) edge[] node[right] {} (D);
\end{tikzpicture}$$
commutes, we get $\Pi_1(\calG,\Delta) \leq \Pi_1(\calG,\Gamma)$.
\end{proof}

With standard arguments one can deduce the following:

\begin{prop}\label{new}
Let $G$ be the fundamental group of a finite reduced graph of profinite groups with at least one edge. If $G$ is not virtually cyclic, then $G$ contains a non-abelian free profinite group. In particular, if $\calC$ does not contains all finite groups, then $G$ is not pro-$\calC$.
\end{prop}

\begin{proof}
Let $U$ be an open normal subgroup and $\til{U}$ as defined in the previous lemma. Then $G_U = G/\til{U}$ is the fundamental group $\Pi_1(\calG_U,\Gamma)$ of a finite graph $(\calG_U,\Gamma)$ of finite groups $\calG(m)\til{U}/\til{U}$ with $\partial_i$ induced by $\partial_i$ of $(\calG,\Gamma)$ (note that the order of $\calG(m)\til{U}/\til{U}$ is bounded by the order of $\calG(m)/(\calG(m) \cap U)$ which is finite since $U$ is open in $G$). As $G = \varprojlim_U G_U$, $G/\til{U}$ is not virtually cyclic for some $U$ and also we can choose such $U$ such that $(\calG_U,\Gamma)$ is reduced. But $G_U$ is the profinite completion of the abstract fundamental group $\abs{\Pi_1}(\calG_U,\Gamma)$ (cf. \cite[Proposition 6.5.1]{Rib}) and $\abs{\Pi_1}(\calG_U,\Gamma)$ is not virtually cyclic, so $\abs{\Pi_1}(\calG_U,\Gamma)$ contains a non-abelian free subgroup of finite index and therefore $G/\til{U}$ contains a non-abelian open free profinite group. Then $G$ contains a non-abelian free profinite group and so, cannot be pro-$\calC$.
\end{proof}

\begin{df}[Standard graph]
Given a finite graph of pro-$\calC$ groups $(\calG,\Gamma)$ with fundamental group $G$, the corresponding {\it standard graph} is
$$S = S(G) = \dot{\bigcup_{m \in \Gamma}} G/\calG(m).$$
The vertices of $S$ are the cosets $g\calG(v)$ with $v \in V(\Gamma)$ and $g \in G$; the edges of $S$ are the cosets $g\calG(e)$ with $e \in E(\Gamma)$ and $g \in G$; the incidence maps of $S$ are given by
$$d_0(g\calG(e)) = g\calG(d_0(e)),\quad d_1(g\calG(e)) = gt_e\calG(d_1(e))$$
with $e \in E(\Gamma)$ and $t_e = 1$ if $e \in E(T)$.
\end{df}

The fact that $S$ is indeed a profinite graph as well as the basic properties of $S$ are explained in \cite[Section 6.3]{Rib}. In particular, it is proved in \cite[Theorem 6.3.5]{Rib} that $S$ is $\calC$-simply connected and so is a pro-$\calC$ tree. Besides,
 $G$ acts continuously on $S$ with $G \bl S \simeq \Gamma$. 

\subsection{Auxiliary Results}\label{subs2.3}

In this subsection we will prove some facts relevant to the general theory of profinite groups acting on profinite graphs. The results here will be helpful to show that $k$-acylindrical actions are well-behaved under taking quotients in the next section.

\begin{lemma}\label{translated graph}
Let $G$ be a profinite group acting on a simply connected profinite graph $\Gamma$ and $\Sigma$ a finite connected subgraph of $\Gamma$. Then $H = \cic{G_v \mid v \in V(\Sigma)}$ is the fundamental group of a finite tree of pro-$\calC$ groups $\Pi_1(\calH, \Delta)$ whose edge and vertex groups are stabilizers of the corresponding edges and vertices of a connected transversal of $\Delta = H \bl H\Sigma$ in $T$.
\end{lemma}

\begin{proof}
Put $T = H\Sigma$. Since the quotient of $T$ modulo every open subgroup $U$ of $H$ is connected, $T$ is connected. By \cite[Proposition 3.7.3]{Rib}, $\Sigma$ is simply connected. By \cite[Proposition~3.9.2]{Rib}, $\Delta=H \bl T$ is simply connected and so is a finite tree. By \cite[Theorem 6.6.1]{Rib}, $H=\Pi_1(\calH, \Delta)$ whose edge and vertex groups are stabilizers of the corresponding edges and vertices of a connected transversal of $\Delta$ in $T$.
\end{proof}

\begin{prop}\label{prop1}
Let $G$ be a profinite group acting on a simply connected profinite graph $\Gamma$ and let $e$ be an edge of $S$. Suppose the stabilizers of the vertices $v$, $w$ of the edge $e$ satisfy the following condition: $|G_v:G_e| + |G_w:G_e| > 4$. Then the group $H = \cic{G_v, G_w}$ is a free profinite amalgamated product $G_v \amalg_{G_e} G_w$. In particular, if $\calC$ does not contains all finite groups, then $H$ is not pro-$\calC$.
\end{prop}

\begin{proof}
By Lemma \ref{translated graph}  
$$H=\Pi_1(\calH,H \bl D) \simeq G_v \amalg_{G_e} G_w.$$
Then, by Proposition \ref{new}, $H$ contains a non-abelian free profinite group and so cannot be pro-$\calC$.
\end{proof}

The previous result does not hold if there is some $e$ such that $|G_v:G_e| = |G_w:G_e| = 2$. When it occurs we will say that $\cic{G_v,G_e,G_w}$ is of {\it dihedral type} or that the edge $e$ {\it generates a dihedral type group}.

\begin{coro}\label{coro1}
Let $\calC$ be a variety  not containing all finite groups and $G$ a pro-$\calC$ group acting on a simply connected profinite graph $\Gamma$. Suppose that $G$ has no edges $e$ generating a dihedral type group. Then for every edge $e$ with vertices $v,w$ and non-trivial edge stabilizer, we have $G_v = G_e$ or $G_w = G_e$.
\end{coro}

\begin{proof}
Suppose that, for a non-trivial edge stabilizer $e$ with vertices $v$ and $w$, we have $G_v$ and $G_w$ both different from $G_e$. By hypothesis, the index of $G_e$ in $G_v$ or in $G_w$ is greater than $2$. By Proposition \ref{prop1}, $H = \cic{G_v,G_w}$ is a not a pro-$\calC$ subgroup of $G$, a contradiction.
\end{proof}

\begin{lemma}\label{subgroup generation in amalgam}
Let $G=G_1\amalg_H G_2$ be a splitting of a profinite group $G$  as an amalgamated free profinite product of profinite groups $G_1,G_2$ and  $H_1\leq G_1, H_2\leq G_2$ be subgroups such that $H_1\cap H \leq U \geq H\cap H_2$ for some open normal subgroup $U$ of $G$. Then  $\langle H_1,H_2\rangle=L_1 \amalg_K L_2$ with $L_1\leq H_1U, L_2\leq H_2U, K\leq H\cap U$.
\end{lemma}

\begin{proof}
By \cite[Proposition 4.4]{ZM-89}, $\langle H_1,H_2\rangle=L_1 \amalg_K L_2$ with $L_1\leq G_1$, $L_2\leq G_2$, $K\leq H$. Recall that the fundamental group $\Pi_1(\calG, \Gamma)$ of a finite graph of profinite groups is a completion of the abstract fundamental group $\abs{\Pi_1}(\calG, \Gamma)$ of the same graph of groups with respect to to the topology generated by the finite index normal subgroups $N$ of $\abs{\Pi_1}(\calG,\Gamma)$ such that $N \cap \calG(v)$ is open in $\calG(v)$. Note that $G$ is the fundamental group of a finite graph of profinite groups. We can consider $U \cap G_1$ and $U \cap G_2$ to apply \cite[Lemma 3.1]{CZ} on $\abs{G} = G_1 \ast_H G_2$ to deduce that $\abs{\cic{H_1,H_2}} = L_1 \ast_K L_2$ where $L_1 \leq H_1(U \cap G_1)$, $L_2 \leq H_2(U \cap G_2)$ and $K \leq H \cap U$. Then   $\cic{H_1,H_2} = \overline L_1 \amalg_K \overline L_2$ with $\overline L_1 \leq H_1U$, $\overline L_2 \leq H_2U$ and $K \leq H \cap U$.
\end{proof}

\begin{coro}\label{amalgam}
Let $G=G_1\amalg_H G_2$ be a splitting of a  profinite group $G$  as an amalgamated free profinite product of profinite  groups $G_1,G_2$ and $H_1 \leq G_1, H_2 \leq G_2$ be subgroups such that $H_1\cap H = 1 = H_2\cap H$. Then $\langle H_1,H_2\rangle=H_1 \amalg H_2$.
\end{coro}

\begin{proof}
Since we have $H_1 \cap H = 1 = H_2 \cap H$ we deduce from Lemma \ref{subgroup generation in amalgam} that $K\leq U$, $L_1 \leq H_1U$, $L_2 \leq H_2U$ for an arbitrary open normal subgroup $U$. Since the intersection of all such $U$ is trivial,  $L_1 = H_1$, $L_2 = H_2$, $K = 1$ and the result follows.
\end{proof}

Let $G$ be a pro-$\calC$ group acting on a pro-$\calC$ tree $T$. Recall that, for any subgroup $U$ of $G$, we set
$$\widetilde{U} = \cic{U \cap G_v : v \in V(T)}.$$

\begin{prop}\label{D compact}
Let $G$ be a pro-$\calC$ group acting on a pro-$\calC$ tree  $T$ and $U$ an open normal subgroup of $G$. Then $D_U=\{m\in \widetilde{U} \bl T : (G/\widetilde U)_m\neq 1\}$ is a profinite subgraph of $\widetilde{U} \bl T$.
\end{prop}

\begin{proof}
Put $G_U=G/\widetilde U$, $T_U=\widetilde{U} \bl T$. It suffices to prove that $D_U$ is closed in $T_U$. As $\{(G_U)_m : m\in T\}$ is a continuous family (see \cite[Lemma 5.2.2]{Rib}), the set $\{m\in \til{U} \bl T : (G/\widetilde U)_m= 1\}=\{m\in \til{U} \bl T : (G/\widetilde U)_m\leq U/\widetilde U\}$ is open. Hence $D_U=\{m\in \widetilde{U} \bl T : (G/\widetilde U)_m\neq 1\}$ is closed.
\end{proof}

\subsection{Relatively projective groups}

As was already mentioned in the introduction, the profinite analogue of the Kurosh Subgroup Theorem does not hold for free profinite products of profinite groups and so the algebraic structure of closed subgroups of free profinite products is not clear. To compensate this, Dan Haran (see \cite{Har87}) introduced the notion of relatively projective  profinite groups and proved that a subgroup $H$ of a free profinite product $\coprod_{x \in X} G_x$ is projective relatively to the family $\{H\cap G_x^g\mid x\in X, g\in G\}$. We shall use a slightly more restricted notion here from \cite{HZ} as it better fits into the profinite version of Bass-Serre theory and suffices for our purpose. 

\begin{df}
A {\it group-pile} is a pair $\bfG = (G,T)$ consisting of a profinite group $G$, a profinite space $T$ and a continuous action of $G$ on $T$. We say that $\bfG$ is {\it finite} if both $G$ and $T$ are finite.
\end{df}

\begin{df}
A {\it morphism} of group-piles $\alpha: (G,T) \to (G',T')$ consists of a group homomorphism $\alpha: G \to G'$ and a continuous map $\alpha: T \to T'$ such that $\alpha(g t) = \alpha(g)\alpha(t)$ for all $t \in T$ and $g \in G$. The above morphism is an {\it epimorphism} if $\alpha(G) = G'$, $\alpha(T) = T'$ and for every $t' \in T'$ there is $t \in T$ such that $\alpha(t) = t'$ and $\alpha(G_t) = {G'}_{t'}$. It is {\it rigid} if $\alpha$ maps $G_t$ isomorphically onto ${G'}_{\alpha(t)}$, for every $t \in T$, and the induced map of the orbit spaces $G \bl T \to G' \bl T'$ is a homeomorphism.
\end{df}

\begin{df}
An {\it embedding problem} for a group-pile $\bfG$ is a pair $(\varphi,\alpha)$ where $\varphi: \bfG \to \bfH$ and $\alpha: \bfG' \to \bfH$ are morphisms of group-piles such that $\alpha$ is an epimorphism. It is {\it finite} if $\bfG'$ is finite. It is {\it rigid} if $\alpha$ is rigid. A {\it solution} to the embedding problem is a morphism $\gamma: \bfG \to \bfG'$ such that $\alpha \circ \gamma = \varphi$.
\end{df}

It means that the diagram
$$\begin{tikzpicture}[auto]
\node(X) at (0,0) {$(G',T')$};
\node(G) at (3,0) {$(H,X)$};
\node(F) at (3,2) {$(G,T)$};
\path[dashed,<-] (X) edge [] node[above left] {$\gamma$} (F);
\path[->] (X) edge [] node[below] {$\alpha$} (G);
\path[<-] (G) edge [] node[right] {$\varphi$} (F);
\end{tikzpicture}$$
is commutative.

For the next definition, let $\calG$ be a continuous family of subgroups of $G$ and $\calC$ a class of finite groups closed for subgroups, quotients and extensions.

\begin{df}
A pile $\bfG$ is {\it $\calC$-projective} if every finite rigid $\calC$-embedding problem for $\bfG$ has a solution. A group $G$ is {\it $\calG$-projective} if every finite $\calC$-embedding problem has a weak solution, that is, there exists a group homomorphism $\gamma: G \to G'$ such that $\alpha \circ \gamma = \varphi$.
\end{df}

We will make use of the following result:

\begin{theo}{\cite[Theorem 4.1]{Zal22}}\label{piletheo}
\begin{enumerate}[(i)]
\item Let $G$ be a prosoluble group acting on a profinite tree $D$ with trivial edge stabilizers. Then $(G,V(D))$ is $\calS$-projective, where $\calS$ is the class of all finite soluble groups.

\item Let $\calC$ be a class of finite groups closed for subgroups, quotients and extensions. Let $(G,T)$ be a projective pro-$\calC$ pile. Suppose that exists a continuous section $\sigma: G \bl T \to T$. Then $G$ acts on a pro-$\calC$ tree with trivial edge stabilizers and non-trivial vertex stabilizers being stabilizers of points in $T$.
\end{enumerate}
\end{theo}

We also need the following proposition:

\begin{prop}{\cite[Proposition 2.5]{Zal22}}\label{alternative}
Let $\calC$ be a class of finite groups closed for subgroups, quotients and extensions. Let $(G,T)$ be a $\calC$-projective group-pile such that there exists a continuous section $\sigma: G \bl T \to T$. Then $G$ embeds into a free pro-$\calC$ product
$$\coprod_{s \in Im(\sigma)} G_s \amalg F$$
where $F$ is a free pro-$\calC$ group of rank $d(G)$, the minimal number of generators of $G$. Moreover, every $G_t$ is of the form $G_s^g \cap G$ for some $g \in \coprod_{s \in Im(\sigma)} G_s \amalg F$, $s \in Im(\sigma).$
\end{prop}

\section{Acylindrical actions and prosoluble groups}\label{sec3}

In this section we prove a generalization of \cite[Lemma 11.2]{WZ17}. This result allows us to deduce that if $\calC$ does not contain all finite groups and if  $G$ is a pro-$\calC$ group acting $k$-acylindrically on a simply connected profinite graph $T$, then for any closed normal subgroup $N$ generated by vertex stabilizers the quotient group $G/N$ acts $k$-acylindrically on $N\backslash T$.

\begin{df}\label{def1}
An action of a profinite group $G$ on a profinite graph $T$ is called {\it $k$-acylindrical} if the stabilizer of any geodesic in $T$ of length greater than $k$ is trivial.
\end{df}

This definition is sometimes stated in the following way:

\begin{df}\label{def2}
An action of a profinite group $G$ on a profinite tree $T$ is called {\it $k$-acylindrical} if, for every nontrivial $g \in G$, the subtree of fixed points
$$T^g = \{m \in T : gm = m\}.$$
has diameter at most $k$.
\end{df}

It is easy to see that both definitions are equivalent: suppose that Definition \ref{def1} holds and there is a path $[v,w]$ with length greater than $k$ in $T^g$ for some nontrivial $g \in G$. Then $g \in G_{[v,w]}$, a contradiction. Conversely, if $T^g$ has diameter at most $k$ for every nontrivial $g \in G$ and $[v,w]$ is a geodesic of length greater than $k$. If $g \in G_{[v,w]}$, then $gv = v$ and $gw = w$ so that $[v,w] \in T^g$, a contradiction.

\begin{df}
We say that a finite graph of profinite groups $(\calG,\Gamma)$ is {\it $k$-acylindrical} if its  profinite fundamental group acts $k$-acylindrically on its standard profinite tree.
\end{df}

\begin{prop}\label{dtprop}
Let $(\calG,\Gamma)$ be a $k$-acylindrical finite graph of profinite groups and $G$ its fundamental group. Then there is no edge $e$ such that
$$|G_{d_0(e)}: G_e| = 2 = |G_{d_1(e)}: G_e|$$
with $G_e \neq 1$.
\end{prop}

\begin{proof}
Assume that there is at least one edge $e$ with
$$|G_{d_0(e)}:G_e| = 2 = |G_{d_1(e)}:G_e|$$
and $G_e \neq 1$. Put $v = d_0(e)$, $w = d_1(e)$ and set $H = \cic{G_v,G_w}$. Let $D = H(v \cup e \cup w)$. Since for every open subgroup $U$ of $H$  the  quotient of $U\backslash D$ is connected, so is $D$. By Lemma \ref{translated graph}  $H = G_v \amalg_{G_e} G_w$. Since $G_e$ is normal in $H$, by \cite[Proposition 4.2.2, (b)]{Rib}, $G_e$ acts trivially on $D$, contradicting the $k$-acylindricity of the action. This finishes the proof.
\end{proof}

This proposition shows that if the action is $k$-acylindrical, then we can apply the results in Subsection \ref{subs2.3} without any restrictions.

\medskip
Recall that a profinite graph $\Gamma$ is also an abstract graph. Let $G$ be a profinite group acting on it.  Put $D=\{e\cup d_0(e)\cup d_1(e)\mid G_e\neq 1\}$. Then $D$ is an abstract  subgraph of $\Gamma$ and is maximal among abstract subgraphs of $\Gamma$ having non-trivial edge stabilizers for all its edges. 

\begin{theo}\label{new3}
Suppose $\calC$ does not contain all finite groups and let $G$ be a pro-$\calC$ group acting $k$-acylindrically on a simply connected profinite graph $\Gamma$. Then the maximal abstract subgraph $D$ of $\Gamma$ such that each $e \in D$ satisfies $G_e \neq 1$ has finite diameter $\leq 2k$. Moreover, for every connected component $D^*$ of $D$ there at most one other connected component at a finite distance to $D^*$ in which case stabilizers of edges and vertices of $D^*$ are of order $2$. 
\end{theo}

\begin{proof}
Suppose on the contrary there exists a finite connected subgraph $\Sigma$ of $D$ whose diameter is $> 2k$. Let $H = \cic{G_v \mid v \in V(\Sigma)}$ and $T = H\Sigma$.  By Lemma \ref{translated graph} and \cite[Proposition 4.1.1]{Rib}, $H$ is the fundamental group of a tree of pro-$\calC$ groups $\Pi_1(\calH, H \bl T)$ whose edge and vertex groups are stabilizers of the corresponding edges and vertices of a connected transversal of $H \bl T$ in $T$. Hence, by Proposition \ref{dtprop} there is no edge $e$ such that
$$|\calH(d_0(e)):\calH(e)| = 2 = |\calH(d_1(e)):\calH(e)|$$
with $G_e \neq 1$. Let $(\calH_{red},H \bl T_{red})$ be a reduced tree of groups obtained from $(\calH, H \bl T)$. If it has more than one vertex, then $H$ cannot be pro-$\calC$ by Proposition \ref{new}. Hence, $(\calH_{red},H \bl T_{red})$ has one vertex $v$ only which implies that $H$ fixes a vertex $w$ in $T$ and, in fact, $w \in \Sigma$. As the diameter of $\Sigma$ is greater than $2k$, it follows that there exists a geodesic $[u,w]$ in $\Sigma$ of length greater than $k$ which is fixed by $H_u \neq 1$ (see \cite[Corollary 4.1.6]{Rib}) contradicting the $k$-acylindricity of the action. This proves the first part of the statement.

Now if there exist a connected component $D_0$ at finite distance from $D^*$ and $v\in D^*,w\in D_0$ are vertices with  $G_v\not\cong C_2$, then $\langle G_v, G_w\rangle$ is not virtually cyclic and so by Proposition \ref{new} can not be pro-$\cal C$, a contradiction. Similarly, if there exists still another connected component $D_1$ at finite distance  from $D^*$ and $u\in V(D_1)$, then $\langle G_v, G_u, G_w \rangle$ is not virtually cyclic again, contradicting Proposition \ref{new} again. This finishes the proof.
\end{proof}

\begin{coro}\label{trivial edge}
Suppose we are in the situation of Theorem \ref{new3} and $\Gamma$ has infinite diameter. Then $\Gamma$ possesses an edge $e$ with $G_e = 1$, namely $G_e=1$ for any $e\not\in D$.
\end{coro}

\begin{coro}\label{G-collapse}
Suppose that $G\backslash D$ has finitely many connected components. Then $G$ acts on a simply connected profinite $G$-quotient graph $T$ of $\Gamma$ with trivial edge stabilizers and vertex stabilizers equal to the stabilizers of some vertices $\Gamma$. Moreover, the number of maximal (by inclusion) $G$-stabilizers of vertices in $T$, up to conjugation, equals to the number of maximal $G$-stabilizers of vertices in $\Gamma$ and this number is bounded by the number of connected components of $G \bl D$.
\end{coro}

\begin{proof}
By \cite[Lemma 3.2]{ChZ22} the closure of any connected component $\cl{D^*}$ is a profinite subgraph of $\Gamma$ of diameter $\leq 2k$, hence  $\cl{D}=\bigcup_i D_i^*$ has finite diameter, where $D_i^*$ runs via orbit representatives of the connected components of $D$. Moreover, as every connected component $\cl{D}_{\ast}$ of $\cl{D}$ has finite diameter, its stabilizer $\Stab(\cl{D}_{\ast})$ fixes some vertex $v_{\ast} \in V(\cl{D}_{\ast})$ (cf. \cite[Corollary, page 20]{Ser}). Now we can  collapse all connected components of the profinite subgraph $\cl{D}$ of $\Gamma$ (see \cite[Exercise 2.1.11]{Rib};  the resulting profinite quotient  graph $T$ of $\Gamma$ is   $G$-invariant and by \cite[Proposition 3.9.1]{Rib} is simply connected. The edge stabilizers of $T$ are  trivial;  the vertices of $T$ whose stabilizers are non-trivial are exactly the vertices to which the connected components of $D$ were collapsed; thus the  only non-trivial vertex stabilizers of $T$ are stabilizers of the connected component $\Stab(\cl{D}_{\ast}) = G_{v_{\ast}}$ of $\cl{D}$. This concludes the proof.
\end{proof}

\begin{theo}\label{acyact}
Suppose $\calC$ does not contain all finite groups and let $G$ be a pro-$\calC$ group acting $k$-acylindrically on a simply connected profinite graph $\Gamma$ and $U$ a closed normal subgroup of $G$. Then the action of $G_U = G/\til{U}$ on $\til{U} \bl \Gamma$ is $k$-acylindrical.
\end{theo}

\begin{proof}
By Theorem \ref{new3} an abstract subgraph $D$ whose edge stabilizers are non-trivial has diameter at most $k$. Consider the natural action of $G_U$ on $\Gamma_U = \til{U} \bl \Gamma$. We shall prove that the same is true in $\Gamma_U = \til{U} \bl \Gamma$; we use the subscript $U$ for images in $\til{U} \bl \Gamma$.

Let $D$ be a maximal connected abstract subgraph of $\Gamma$ where each edge $e \in D$ satisfies $G_e \neq 1$. Note that any edge $e_U$ not in $D_U$ is an image of an edge $e \not\in D$ and since $G_e$ is trivial, $(G_U)_{e_U}$ is trivial. Therefore, a maximal connected (abstract) subgraph of $\Gamma_U$  whose all edge stabilizers are non-trivial is contained in $D_U$ and so has diameter at most $k$, which implies that the action of $G_U$ on $\Gamma_U$ is $k$-acylindrical.
\end{proof}

Recall that for any open normal subgroup $U$ of a pro-$\calC$ group $G$ acting on a pro-$\calC$ tree $T$ we  set
$$\widetilde{U} = \cic{U \cap G_v : v \in V(T)},$$
and $G_U=G/\widetilde U$, $T_U =\widetilde U\backslash T$.

\begin{prop}\label{compact non-trivial stabilizers} 
Suppose $\calC$ does not contain all finite groups and let $G$ be a pro-$\calC$ group acting $k$-acylindrically on a simply connected profinite tree $T$. The group $G_U$ acts on a profinite simply connected $G_U$-quotient tree $T^{\ast}_U$ of $T_U$ with trivial edge stabilizers and vertex stabilizers equal to the  maximal (by inclusion) stabilizers of  vertices of $T_U$. Moreover, if $G$ is prosoluble then $(G_U, V(T^{\ast}_U))$ is $\mathcal S$-projective, where $\mathcal S$ is the class of all finite soluble groups.
\end{prop}

\begin{proof}
By Proposition \ref{D compact} $D_U=\{m\in \widetilde U \bl T\mid (G/\widetilde U)_m\neq 1\}$ is a profinite subgraph of $T_U$. We can  collapse all connected components of the profinite subgraph $D_U$ of $T_U$ (see \cite[Exercise 2.1.11]{Rib}); the resulting profinite quotient graph $T^{\ast}_U$ of $T_U$ is $G$-invariant and by \cite[Proposition 3.9.1]{Rib} is  simply connected. The edge stabilizers of $T^{\ast}_U$ are trivial;  the vertices of $T^{\ast}_U$ whose stabilizers are non-trivial are exactly the vertices to which the connected components of $D_U$ were collapsed; thus the only non-trivial vertex stabilizers of $T^{\ast}_U$ are stabilizers of the connected component $\Stab(D_U)_{\ast}$ of $D_U$. But as was explained in the proof of Theorem \ref{acyact} any connected component of $D_U$ has diameter $\leq 2k$. Hence $\Stab(D_U)_{\ast}$ stabilizers a vertex in $(D_U)_\ast$ (cf. \cite[Corollary, Page 20]{Ser}) and clearly the stabilizer of this vertex is maximal by inclusion. This concludes the first part of the proof.

Suppose now $G$ is prosoluble. By item (i) of  Theorem \ref{piletheo} $(H_U,V(T^*_U))$ is $\calS$-projective that finishes the proof of second statement. 
\end{proof}

\begin{lemma}\label{kernel not generated by torsion}
Let $G$ be a prosoluble group acting $k$-acylindrically on a simply connected profinite tree $T$. Suppose for some primes $p\neq q$ and  some $v,w\in V(T)$ from different $G$-orbits, the stabilizer $G_v$ has a subgroup of order $p$ and the stabilizer $G_w$ has a subgroup of order $q$ as subquotients. Then there exists a subgroup $H$ of $G$ that admits an epimorphism $f:H \to C_{pq}$ to a cyclic group of order $pq$ such that $|f(H_v)|=p, |f(H_w)|=q$ and $K=\ker(f)$ is not generated by its vertex stabilizers. Moreover, non-trivial subgroups of the images of $H_v$ and $H_w$ in $H/\widetilde K$ do not commute up to conjugation.
\end{lemma} 

\begin{proof} 
W.l.o.g we may assume that $G_v$ and $G_w$ are maximal (by inclusion) vertex stabilizers. Choose an open normal subgroup $U$ of $G$ such that  $G_vU/U$ and $G_wU/U$  have subgroups $G_p$ of order $p$ and $G_q$ of order $q$, respectively. Consider the action of $G_U=G/\widetilde U$ on $T_U=\widetilde U\backslash S$ and note that in $G_U$ the vertex stabilizers of the images $v_u, w_u$ of $v$ and $w$ in $T_U$ are  $G_vU/U$ and $G_wU/U$ (still in different orbits) and refining $U$ we  may assume that $G_vU/U$ and $G_wU/U$ are maximal vertex stabilizers.   By Proposition \ref{compact non-trivial stabilizers} $G_U$ acts on a simply connected (infinite) profinite tree $T^*_U$ with trivial edge stabilizers, the vertex stabilizers of $T^*_U$ are maximal vertex stabilizers of $T_U$  and  $(G_U, V(T^*_U))$ is $\mathcal S$-projective, where $\calS$ is the class of all finite soluble groups. Hence $G_p$ and $G_q$ are subgroups of some  stabilizers of vertices of $T^*_U$ from distinct $G_U$-orbits.  Note that if a subgroup of $G$ generated by its vertex stabilizers then its image in $G_U$ generated by its vertex stabilizers as well, and vertex stabilizers in $G_U$ are finite. Therefore to prove the lemma it suffices to construct a subgroup $H$ in $G_U$ containing $G_p$ and $G_q$ and an epimorphism $f:H\to C_{pq}$ to a cyclic group of order $pq$ with $|f(G_p)|=p, |f(G_q)|=q$ such that $\ker(f)$ is not generated by torsion.

Put $H=\langle G_p,G_q\rangle$. Then by \cite[Lemma 4.8]{Zal23} setting $$\calH = \{H \cap (G_U)_v \mid v \in V(T_U)\}$$
there is a profinite $H$-space $X$ such that
\begin{itemize}
\item [(a)]
$\calH = \{(H)_x \mid x \in X\}$;

\item [(b)] $(H)_x\cap (H)_{x'}=1$ for $x\neq x'$;

\item [(c)] $\bfH = (H,X)$ is a $\calC$-projective pile.
\end{itemize}
Moreover, since $H$ is finitely generated, it is second countable and so by \cite[Remark 4.9]{Zal23} there exists a second countable such $X$. Therefore there exists a continuous section $\sigma:H\backslash X\to X$ (see \cite[Lemma 5.6.7]{RZ}). 

 Hence by Proposition  \ref{alternative} $H$ embeds into a free prosoluble product $H_0=\coprod_{s\in Im(\sigma)} H_s \amalg F$, where $F$ is  free prosoluble, such that  each $H_x=H_s^h$ for some $h\in H_0$.  Choose $x_1,x_2\in X$ such that $H_{x_1}$  contains $G_p$ and $H_{x_2}$  contains $G_q$, observe that $x_1,x_2$ are from different $H$-orbits and let $s_1\neq s_2\in Im(\sigma) $ such that $H_{s_1}$ and $H_{s_2}$ conjugate to $H_{x_1}$ and $H_{x_2}$ respectively. 
 
 By \cite[Theorem 5.3.4]{Rib} $H_0$ is an inverse limit of free prosoluble products with finitely many factors which are quotients free factors of $H_0$ and so there exists such a free prosoluble product $\coprod_{i=1}^n H_{s_i} \amalg F$ with finitely many factors, where $H_{s_1}$ and $H_{s_2}$ are still the free factors. Then factoring out all factors but $H_{s_1}$ and $H_{s_2}$ we arrive at a free prosoluble product $H_{s_1}\amalg H_{s_2}$, where $H_{s_1}$ and $H_{s_2}$ are still conjugate to the images of $H_{x_1}$ and $H_{x_2}$ respectively. 
 
 Now define a natural  epimorphism  $H_{s_1}\amalg H_{s_2}\to H_{s_1}\times H_{s_2}$ that sends $H_{s_1}$ onto $H_{s_1}$ and $H_{s_2}$ onto $H_{s_2}$ identically. The kernel of this natural epimorphism is free prosoluble (by  \cite[Theorem 9.1.6]{RZ}). Hence the kernel of $f$ factors through the kernel of the restriction of this epimorphism to the image $M_U$ of $H_U$ in $H_{s_1}\amalg H_{s_2}$ which is torsion-free. Therefore the kernel of $f$ is not generated by torsion as required. Moreover, as non-trivial subgroups of $H_{s_1}, H_{s_2}$ in $H_{s_1}\amalg H_{s_2}$ do not commute up to conjugation (cf. \cite[Theorem 9.1.2]{Rib}),  hence   non-trivial subgroups of $H_{x_1}$ and $H_{x_2}$ of  $M_U$ do not commute up to conjugation as well.
\end{proof}

\section{Closed subgroups of fundamental groups of finite graphs of profinite groups with finite edge groups}\label{sec4}

In this section we will conclude an important step to prove Theorem \ref{ori}. The main theorem of this section gives us necessary conditions for a prosoluble group to be a subgroup of the fundamental group of an injective $k$-acylindrical finite graph of profinite groups with finite edge groups. The fundamental groups of a finite graph of profinite groups are always profinite (or pro-$\calC$) in this section unless otherwise said; if vertex groups are finitely generated they are just profinite completion of the classical Bass-Serre fundamental groups.
\\

This lemma is a general version of \cite[Proposition 2.16]{CZ}:

\begin{prop}\label{free product splitting}
Let $G$ be the fundamental pro-$\calC$ group of an injective reduced finite graph of pro-$\calC$ groups $(\calG,\Gamma)$. Suppose that there is an edge $e$ of $\Gamma$  such that $G(e) = 1$ and $G(d_i(e)) \neq 1$ for $i=0,1$. Then $G$ splits as a free pro-$\calC$ product of either   fundamental groups of two subgraphs of groups of $(\calG,\Gamma)$ or as the fundamental group of the subgraph of groups of $(\calG,\Gamma)$ and $\widehat\Z_\calC$.
\end{prop}

\begin{proof}
Let
$$\begin{tikzpicture}[-latex, auto, node distance=3cm,
semithick]
\node[circle,draw=red, fill=red, inner sep=0pt, minimum size=1.5mm] at (0,0) {};
\node[circle,draw=red, fill=red, inner sep=0pt, minimum size=1.5mm] at (3,0) {};
\node[circle,draw=black, fill=black, inner sep=0pt, minimum size=1.5mm] at (5,1) {};
\node[circle,draw=black, fill=black, inner sep=0pt, minimum size=1.5mm] at (5,-1) {};
\node[circle,draw=black, fill=black, inner sep=0pt, minimum size=1.5mm] at (4,2) {};
\node[circle,draw=black, fill=black, inner sep=0pt, minimum size=1.5mm] at (4,-2) {};
\node[circle,draw=black, fill=black, inner sep=0pt, minimum size=1.5mm] at (-2,1) {};
\node[circle,draw=black, fill=black, inner sep=0pt, minimum size=1.5mm] at (-2,-1) {};
\node[circle,draw=black, fill=black, inner sep=0pt, minimum size=1.5mm] at (-1,2) {};
\node[circle,draw=black, fill=black, inner sep=0pt, minimum size=1.5mm] at (-1,-2) {};
\draw[color=red] node[above right] {} (0,0) -- node[below] {$G(e)$} (3,0) node[above left] {};
\draw (3,0) -- (4,2);
\draw (3,0) -- (5,1);
\draw (3,0) -- (5,-1);
\draw (3,0) -- node[below left] {} (4,-2) node[below right] {};
\node at (5,0) {$\vdots$};
\draw (0,0) -- node[above right] {} (-1,2) node[above left] {};
\draw (0,0) -- (-2,1);
\draw (0,0) -- (-2,-1);
\draw (0,0) -- (-1,-2);
\node at (-2,0) {$\vdots$};
\end{tikzpicture}$$
be a slice of the graph $\Gamma$.  

After removing the edge $e$ 
we have  two possibilities: either the resulting graph is connected or not connected. We shall prove that $G$ splits as a free profinite product in each of the possibilities.

Case 1. After removing the edge $e$ the graph $\Gamma$ become not connected:
$$
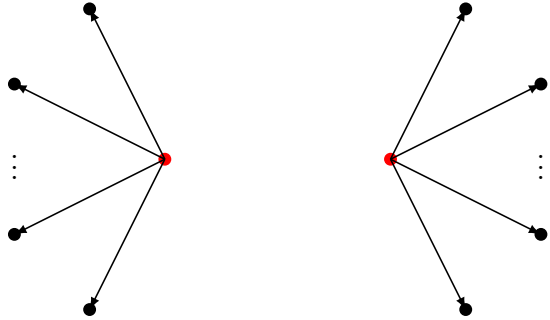
\begin{tikzpicture}[-latex, auto, node distance=3cm,
semithick]
\node[circle,draw=red, fill=red, inner sep=0pt, minimum size=1.5mm] at (0,0) {};
\node[circle,draw=red, fill=red, inner sep=0pt, minimum size=1.5mm] at (3,0) {};
\node[circle,draw=black, fill=black, inner sep=0pt, minimum size=1.5mm] at (5,1) {};
\node[circle,draw=black, fill=black, inner sep=0pt, minimum size=1.5mm] at (5,-1) {};
\node[circle,draw=black, fill=black, inner sep=0pt, minimum size=1.5mm] at (4,2) {};
\node[circle,draw=black, fill=black, inner sep=0pt, minimum size=1.5mm] at (4,-2) {};
\node[circle,draw=black, fill=black, inner sep=0pt, minimum size=1.5mm] at (-2,1) {};
\node[circle,draw=black, fill=black, inner sep=0pt, minimum size=1.5mm] at (-2,-1) {};
\node[circle,draw=black, fill=black, inner sep=0pt, minimum size=1.5mm] at (-1,2) {};
\node[circle,draw=black, fill=black, inner sep=0pt, minimum size=1.5mm] at (-1,-2) {};
\draw (3,0) -- (4,2);
\draw (3,0) -- (5,1);
\draw (3,0) -- (5,-1);
\draw (3,0) -- node[below left] {} (4,-2) node[below right] {};
\node at (5,0) {$\vdots$};
\draw (0,0) -- node[above right] {} (-1,2) node[above left] {};
\draw (0,0) -- (-2,1);
\draw (0,0) -- (-2,-1);
\draw (0,0) -- (-1,-2);
\node at (-2,0) {$\vdots$};
\end{tikzpicture}$$
Let $\Gamma_1, \Gamma_2$ be the connected components and $G_1, G_2$ its fundamental groups, respectively. They are embedded in $G$ by Lemma \ref{aux}. It follows that $G = G_1 \amalg G_2$. 

\medskip
Case 2. After removing $e$ the graph remains connected
$$\begin{tikzpicture}[-latex, auto, node distance=3cm,
semithick]
\node[circle,draw=green, fill=green, inner sep=0pt, minimum size=1.5mm] at (0,0) {};
\node[circle,draw=green, fill=green, inner sep=0pt, minimum size=1.5mm] at (3,0) {};
\node[circle,draw=green, fill=green, inner sep=0pt, minimum size=1.5mm] at (5,1) {};
\node[circle,draw=green, fill=green, inner sep=0pt, minimum size=1.5mm] at (5,-1) {};
\node[circle,draw=green, fill=green, inner sep=0pt, minimum size=1.5mm] at (4,2) {};
\node[circle,draw=green, fill=green, inner sep=0pt, minimum size=1.5mm] at (4,-2) {};
\node[circle,draw=green, fill=green, inner sep=0pt, minimum size=1.5mm] at (-2,1) {};
\node[circle,draw=green, fill=green, inner sep=0pt, minimum size=1.5mm] at (-2,-1) {};
\node[circle,draw=green, fill=green, inner sep=0pt, minimum size=1.5mm] at (-1,2) {};
\node[circle,draw=green, fill=green, inner sep=0pt, minimum size=1.5mm] at (-1,-2) {};
\path[dashed] (-1,2) edge[color=green,bend left] node[] {} (4,2);
\draw[color=green] (3,0) -- (4,2);
\draw (3,0) -- (5,1);
\draw (3,0) -- (5,-1);
\draw[color=green] (3,0) -- node[below left] {} (4,-2) node[below right] {};
\node at (5,0) {$\vdots$};
\draw[color=green] (0,0) -- node[above right] {} (-1,2) node[above left] {};
\draw (0,0) -- (-2,1);
\draw (0,0) -- (-2,-1);
\draw[color=green] (0,0) -- (-1,-2);
\node at (-2,0) {$\vdots$};
\end{tikzpicture}$$
then we can choose a maximal subtree $T$ not containing $e$ and use it to calculate the fundamental group $G_1$ of the graph of groups restricted to $\Gamma - \{e\}$. So, $G$ is the fundamental group of the loop with edge correspondent to $G(e)$ and vertex correspondent to $G_1$ so that $G = \HNN(G_1,G(e),t) = G_1 \amalg \cic{t}$, since $G(e) = 1$.
$$\begin{tikzpicture}[]
\node[circle,draw=black, fill=black, inner sep=0pt, minimum size=1.5mm] at (0,0) {};
\path[->] (0,0) edge[color=red,scale=6,in=-150,out=-210,loop] node[left] {$G(e)$} (0,0) node[right] {$G_1$};
\end{tikzpicture}$$
\end{proof}

\begin{theo}\label{subfreeprod}
Let $(\calG,\Gamma)$ be an injective $k$-acylindrical finite graph of profinite groups with finite edge groups and $G$ its profinite fundamental group. Suppose  $\calC$ does not contain all finite groups. If $H$ is a pro-$\calC$ subgroup of $G$, then there exists $e\in E(\Gamma)$ and a conjugate $H_0$ of $H$ such that $H_0\cap G(e)=1$ and for any such $e$ 
$$H_0 \leq U_1 \amalg U_2$$
where $U_1 \simeq \Pi_1(\calU,\Delta)$   is the profinite fundamental group of an injective $k$-acylindrical finite graph  of profinite groups  with  finite edge groups, $\Delta$  a connected component of  $\Gamma\setminus \{e\}$ and  $U_2$ is either isomorphic to $\widehat{\Z}$ or similarly is the profinite fundamental group of an injective $k$-acylindrical finite graph of profinite groups with finite edge groups.
\end{theo}

\begin{proof}
Consider the action of $H$ on the standard profinite tree $S=S(G)$. By Corollary \ref{trivial edge} there is an edge $e$ of $S$ for which $H_{e}$ is trivial. Then  there is an open subgroup $U_H$ of $G$ containing $H$ such that $(U_H)_{e} = 1$. By \cite[Corollary 4.5]{ZM-89} $U_H$ is the profinite fundamental group of a finite graph of groups $(\calU, \Delta)$ whose vertex and edge groups are $U_H$-stabilizers of vertices and edges of $S$. Thus one of the edge groups $U(\bar e)$ of $(\calU, \Delta)$ is trivial. Assuming, w.l.o.g., that $(\calU, \Delta)$ is reduced we can apply Proposition \ref{free product splitting} to deduce that $U_H$ splits as a free profinite product. As $U(\bar e)$ is conjugate by some element $g\in G$ into some edge group of $(\calG,\Gamma)$ we may conjugate $H$ and $U$ by this element to get the required $H_0$. This finishes the proof.
\end{proof}

The second author \cite[Theorem 1.1]{Zal23} described the prosoluble subgroups of free profinite products. Using this we can deduce the following

\begin{coro}\label{coro2}
One of the following holds for $H_0$ if it is prosoluble:
\begin{enumerate}[(i)]
\item There is a prime $p$ such that $H_0 \cap u U_i u^{-1}$ is pro-$p$ for every $i=1,2$ and $u \in U_1\  \amalg\  U_2$; besides if $H$ is finitely generated then $\sum_{i,u} d(H_0 \cap u U_i u^{-1})\leq d(H)$.

\item $H_0 \simeq \widehat{\Z}_{\pi} \rtimes C$ is a profinite Frobenius group.

\item A conjugate of $H_0$ is a subgroup of $U_i$ for some $i=1,2$.
\end{enumerate}
\end{coro}

If $H$ is torsion-free, then item $(ii)$ of Corollary \ref{coro2} does not occur. 

\bigskip
It is not difficult to construct an example satisfying Theorem \ref{subfreeprod}.

\begin{ex}
Let $G_1, G_2$ be profinite Frobenius groups (or any finite Frobenius groups) with isomorphic complement $C$. Consider the group $G = G_1 \amalg_{C} G_2$. Then $C$ is malnormal in $G$ and so $G$ satisfies the hypothesis of Theorem \ref{subfreeprod}. Then any prosoluble subgroup $H$ has the structure stated in Corollary \ref{coro2}. 
\end{ex}

The next proposition is of independent interest.

\begin{prop}\label{distance}
Let $G=\Pi_1(\calG,\Gamma)$ be the profinite fundamental group of a $k$-acylindrical finite graph of profinite groups. Let $v,w \in V(\Gamma)$ be vertices at distance at least $2k+1$. Then $\cic{G(v), G(w)}=G(v)\amalg G(w)$. 
\end{prop}

\begin{proof}
Let $[v,w]$ be a shortest path between $v$ and $w$. Let $G(v,w)$ be the fundamental group of the finite graph of profinite groups restricted to $[v,w]$. By Lemma \ref{aux} $G(v,w)$ is a subgroup of $G$ generated by vertex groups of $[v,w]$.  Let $e$ be an edge of $[v,w]$  at distance $>k$ from   $w$ and $v$. Then $G(v)\cap G(e)=1=G(w)\cap G(e)$. Note that $G(v,w)$ splits over $G(e)$ as a free amalgamated profinite product $G(v,w)=G_1\amalg_{G(e)} G_2$, where $G_1$, $G_2$ are  profinite groups generated by vertex groups of the connected components  of $[v,w]\setminus \{e\}$ (see Lemma \ref{aux}), so that $G(v)\leq G_1$ and $G(w)\leq G_2$. By Corollary \ref{amalgam}, $\cic{G(v),G(w)}=G(v)\amalg G(w)$ as required.
\end{proof}

For the proof of the next theorem we shall need two simple lemmas.

\begin{lemma}\label{intersection tildes}
Let $G$ be a pro-$\calC$ group acting on a pro-$\calC$ tree  $T$. Let $\calU=\{U_i\}$ be a filtered from below family of subgroups of $G$ with $H=\bigcap_{U\in \calU} U$. Then $\widetilde H=\bigcap_{U\in \calU} \widetilde U$.\end{lemma}

\begin{proof}
By \cite[Proposition 4.1.1]{Rib} $\widetilde U \bl T$ is a pro-$\calC$ tree and so
$$\left(\bigcap_{U\in \calU}\widetilde{U}\right) \bl T=\varprojlim_{U\in \calU} \widetilde U \bl T$$
is a pro-$\calC$ tree (cf. \cite[Proof of Lemma 4.2.7]{Rib}). Hence by \cite[Proof of Lemma 4.2.7]{Rib} $\bigcap_{U\in \calU} \widetilde U$ is generated by the vertex stabilizers and so equal to $\widetilde H$.
\end{proof}

\begin{lemma}\label{C_{pq}}
Let $G$ be an extension  of an abstract free group $F$ by a group $C_{pq}$ of order $pq$ for primes $p\neq q$. Suppose that $G$ contains  finite subgroups  $A$ and $B$  of order $p$ and $q $ that  do not  commute up to conjugation. Then the natural epimorphism $f:G\to C_{pq}$ with kernel $F$ factors through an abstract free product  $ C_p \ast C_q$.\end{lemma}

\begin{proof} By  the results of Karrass, Pietrowski, Solitar, Cohen and Scott \cite{KPS73, Coh06, Sco74}) $G$ splits  as the fundamental group $\Pi_1(\calG, \Gamma)$ of a finite graph of finite groups and assuming w.l.o.g. that $(\calG, \Gamma)$ is reduced, one sees that the vertex groups of $(\calG,\Gamma)$ are of order at most $pq$ and edge groups are of order at most $p$ or $q$. 

As any finite group is conjugate into a vertex group, we shall assume w.l.o.g. that $A$ and $B$ are in the vertex groups of $(\calG, \Gamma)$. Let $(\calG, \Gamma_A)$ be the maximal subgraph of groups of $(\calG,\Gamma)$ containing $A$ in its every vertex group and $(\calG, \Gamma_B)$ is the maximal subgraph of groups of $(\calG,\Gamma)$ containing $B$ in its every vertex group. Since $A$ and $B$ do not commute up to conjugation, $\Gamma_A$ and $\Gamma_B$ do not intersect. Then  factoring out the normal closure of all subgroups of order $q$ contained in the vertex groups of $(\calG, \Gamma_A)$ and all subgroups of order $p$ contained in the vertex groups of $(\calG, \Gamma_B)$ we obtain the fundamental group $\Pi_1(\calH, \Gamma)$, where $A$ is still contained in a vertex group of  $(\calH, \Gamma_A)$, $B$  is still contained in the vertex groups of  $(\calH, \Gamma_B)$ but neighboring edge groups with trivial edge groups of $\Gamma_A$ and $\Gamma_B$ are trivial. This means that $A$ and $B$ are free factors of $\Pi_1(\calH, \Gamma)$ and so  so the group $\Pi_1(\calH, \Gamma)$ can be mapped to $A \ast B$.
\end{proof}

\begin{theo}\label{pop generalization}
Let $(\calG,\Gamma)$ be a $k$-acylindrical finite graph of profinite groups, $G$ its profinite fundamental group and $S=S(G)$ its standard profinite tree. Let $H$ be a  subgroup of $G$ such that there exist some non-trivial  non-conjugate vertex stabilizers  $H_v$, $H_w$ for $v,w\in S$  which are not both pro-$p$ for any prime $p$. Then $H$ is not prosoluble.
\end{theo}

\begin{proof}
Suppose on the contrary $H$ is prosoluble. Since $H_v$ and $H_w$ are not both pro-$p$ for a fixed prime $p$, there exist primes $p\neq q$ such that  groups of order $p$ and $q$ are subquotients of $H_v$ and $H_w$ respectively.   Then by Corollary \ref{kernel not generated by torsion} there exists an open subgroup $L$ of $H$ that admits an epimorphism $f_L:L\to C_{pq}$ to the cyclic group of order $pq$ with the kernel not generated by  vertex stabilizers such that $f_L(L_v)$ has order $p$ and $f_L(L_w)$ has order $q$. Thus, w.l.o.g., we may assume that $H=L$ and $f_L=f_H$.  By \cite[Lemma 8.3.8]{RZ} $f_H$ extends to an epimorphism $f_U:U\to C_{pq}$ from some open subgroup $U$ of $G$  containing $H$ and by Lemma \ref{intersection tildes} we can choose such $U$ with the kernel $K$ of $f_U$  not generated by its vertex stabilizers. Let $\widetilde K$ be the normal subgroup of $K$ generated by the vertex stabilizers. Then $\widetilde K$ is normal in $U$ and we can consider $U/\widetilde K$  acting on $\widetilde K\backslash S$ which is simply connected by \cite[Lemma 4.7]{ZM-89}. Then by \cite[Proposition 4.4]{ZM-89} $U/\widetilde K$ is the fundamental profinite group of a finite graph of finite groups $\Pi_1(\calU,U\backslash S)$ and so is the profinite completion of the abstract fundamental group $\Pi=\Pi_1^{abs}(\calU, U\backslash S)$ (cf. \cite[Proposition 6.5.6]{Rib}). Note that we have an induced epimorphism $f_K:U/\widetilde K\to C_{pq}$ and we denote by $f^{abs}$ its restriction to $\Pi$. Since $K/\widetilde K$ is projective (see \cite[Corollary 4.1.3]{Rib} or \cite[Theorem 2.6]{ZM88}) and so is torsion free, the kernel $\ker(f^{abs})$ is free and so by Lemma \ref{C_{pq}} $f^{abs}$ factors through the free product $C_p \ast C_q$.  Hence $f_K$ factors through the free profinite product $C_p\amalg C_q=\widehat{C_p \ast C_q}$. Thus we have the following commutative diagram:
 
$$\xymatrix{&& U\ar[d]&\\
           H\ar[rrdd]_{f_H}\ar[urr]\ar@{-->}[rrd]\ar[rr] &&U/\widetilde K\ar[d]&\Pi\ar[l]\ar[d]\\
           &&C_p\amalg C_q\ar[d]&C_p*C_q\ar[l]\ar[ld]\\
                        && C_{pq}&}$$          
 
Since $f_H(H_v)$ has order $p$ and $f_H(H_w)$ has order $q$, it follows from this commutative diagram that  the image of $H$  in $C_p\amalg C_q$ contains a group of order $p$ and a group of order $q$.

Choose now elements $c_p, c_q\in S_n$, $n>4$ of order $p$ and $q$, such that any conjugate of them generate non-soluble subgroups (see \cite[Lemmas 1.1 and 1.2]{Jar94}) and let $f_p:C_p\amalg C_q\to C_p < S_n$, $f_q:C_p\amalg C_q\to C_q < S_n$ be the corresponding epimorphisms. Then we have an epimorphism $\varphi:C_p\amalg C_q\to S_n$. Since the image of $H$ in $C_p\amalg C_q$ contains the group of order $p$ and the group of order $q$, it maps epimorphically to $S_n$ which is not soluble, a contradiction. This finishes the proof.
\end{proof}

\begin{coro}\label{generalizing pop}
Let $(\calG,\Gamma)$ be a $k$-acylindrical finite graph of profinite groups and $G$ its profinite fundamental group. Let $H$ be a  prosoluble subgroup of $G$. Then one of the following holds:
\begin{enumerate}[(i)]
\item all non-trivial intersections $H\cap G(v)^g$, $g\in G, v\in V(\Gamma)$ are pro-$p$;

\item $H \simeq \widehat{\Z}_\pi\rtimes C$ is a profinite Frobenius group;

\item $H\leq G(v)^g$ for some $g\in G, v\in V(\Gamma)$.
\end{enumerate}
\end{coro}

\begin{proof}
Suppose that the non-trivial $H \cap G(v)^g$ are not all pro-$p$ and that there is no $g \in G, v \in V(\Gamma)$ such that $H \leq G(v)^g$. Then by Theorem \ref{pop generalization} all $H\cap G(v)^g$ are conjugate in $H$. Consider the action of $H$ on the standard profinite tree $S=S(G)$. Then $H\cap G(v)^g$ are exactly the stabilizers of vertices of $S$ in $H$. Thus all vertex stabilizers of $H$ are conjugate in $H$ and this is true for any subgroup of $H$. It follows that $D=\{m\in S\mid H_m\neq 1\}$ has only one connected component up to translation and so by Corollary \ref{G-collapse} we can collapse the connected components of $D$ to obtain the action of $H$ on the simply connected profinite graph $T$ with trivial edge stabilizers.  Let $U$ be an open normal subgroup of $H$. Then by Theorem \ref{acyact} $H_U = H/\til{U}$ acts $k$-acylindrically on $T_U=\til{U} \backslash T$  with finite vertex stabilizers and  note that we still have that all vertex stabilizers are conjugate in $H_U$ and again this is true for any subgroup of $H_U$.  Then $(H_U)_v$ is finite soluble, and not a finite $p$-group for almost all $U$. So we can choose in $(H_U)_v$ a subgroup of the form $C=A\rtimes C_l$,  a semidirect product of elementary abelian $p$-group $A$ and a cyclic group of order $l$ for some distinct primes $p, l$. Let $F_U$ be a projective open normal subgroup of $H_U$ and consider $F_UC = F_U \rtimes C$. Since $C$ is self-normalized in $F_UC$ (cf. \cite[Corollary 7.1.6 (a)]{Rib}) it satisfies \cite[Lemma 2.4]{Zal23} and so applying it we deduce that $H_U$ is soluble.  Then by \cite[Theorem 4.2.11]{Rib} $H_U$ is Frobenius. Note that   $H = \varprojlim_U H_U$ and so by \cite[Theorem 4.6.9]{RZ}  $H$ is profinite Frobenius as well.
\end{proof}

\begin{prop}\label{finvs}
Let $(\calG,\Gamma)$ be an injective $k$-acylindrical finite graph of profinite groups, $G$ its profinite fundamental group, $S$ its standard graph and $H$ is a prosoluble subgroup of $G$. If $H$ is finitely generated, then it has only finitely many maximal vertex stabilizers up to conjugation. In particular, if $D$ is the maximal abstract subgraph of $S$ such that each $e \in D$ satisfies $H_e \neq 1$, then $D$ has finitely many connected components up to translation.
\end{prop}

\begin{proof}
By Corollary \ref{generalizing pop} we have that $H$ satisfies (i), (ii) or (iii).

If $H$ satisfies (ii) and not (i) and (iii), then all vertex stabilizers are finite and conjugate (as they are Hall subgroups, cf. \cite[Theorem 2.3.5]{RZ}).

It remains to analyze (i) and (iii). Choosing any open subgroup $U$ of $H$, we can pass to the quotient $H_U = H/\til{U}$ acting on $S_U$. Recall that  $H = \varprojlim_U H_U$ and $S = \varprojlim_U S_U$. So, to prove that the number of maximal vertex stabilizers of $H$ is finite, up to conjugation, it is enough to prove that the number of maximal vertex stabilizers up to conjugation in $H_U$ is bounded independently on  $U$. 

Suppose first that $H$ satisfies (i). Then by Proposition \ref{compact non-trivial stabilizers} and Theorem \ref{acyact} $H_U$ acts acylindrically on a simply connected profinite tree $S^{\ast}_U$  with trivial edge stabilizers, vertex stabilizers equal to the maximal stabilizers of $S_U$ and $(H_U, V(\mathcal{S}^{\ast}_U))$ is $\cal S$-projective. Moreover, by Theorem \ref{new3} the connected components of $D_U=\{e\cup d_0(e)\cup d_1(e)\mid (H_U)_e\neq 1\}$ have finite diameter and so $S^{\ast}_U$ is non-trivial by Corollary \ref{trivial edge}.  By \cite[Proposition 4.9]{Zal23}, the number of non-conjugated $(H_U)_{v_{\ast}}$, $v_{\ast} \in S^{\ast}_U$, is bounded by $d(H_U) \leq d(H)$. Since the non-trivial vertex stabilizers of $S^{\ast}_U$ are stabilizers of connected components of $D_U$, it follows that the number of connected components of $D_U$ and hence the number of the maximal vertex stabilizers of $S_U$ is bounded by $d(H)$  up to conjugation, independently on $U$.

Suppose that $H$ satisfies (iii). If $H \leq G(v)^g$ then $H$ fixes a vertex and the same is true for each $H_U$. It follows that the number of maximal vertex stabilizers, up to conjugation, is trivially bounded by $d(H)$.

Thus, in both cases (i) and (iii) each $H_U$ has finitely many maximal vertex stabilizers, up to conjugation. This finishes the proof.
\end{proof}

{\it Proof of Theorem \ref{sela}} follows from Proposition \ref{finvs}.

\section{The finitely generated case}\label{sec5}

In this section we shall study finitely generated prosoluble subgroups of profinite groups acting $k$-acylindrically on profinite trees.

\begin{reptheorem}{general}
Let $(\calG,\Gamma)$ be an injective $k$-acylindrical finite graph of profinite group and $G$ its fundamental group. If $H$ is a finitely generated prosoluble subgroup of $G$, then
$H$ embeds into a  free profinite product $\coprod_{v \in V(\Gamma)} H_v$ of finitely many finitely generated profinite groups $H_v =H\cap G(v)^g$ for some $g \in G$. Moreover, if $H\neq H_v$ for some $v$, then either $H_v$ is finitely generated pro-$p$ or $H \simeq \widehat{\Z}_{\pi} \rtimes C$ is a profinite Frobenius group.
\end{reptheorem}

\begin{proof}
Note that the abstract subgraph $D=\{m\mid H_m\neq 1\}$ has finite diameter by Theorem \ref{new3}, and by Proposition \ref{finvs} $D$ has finitely many connected components up to translation. Therefore $H \bl D$ has finitely many connected components and so by Corollary \ref{G-collapse} there exists a simply connected $H$-quotient graph $\overline S$ on which $H$ acts with trivial edge stabilizers. 

By item (i) of Theorem \ref{piletheo}, $(H,V(\overline S))$ is an $\calS$-projective group-pile where $\calS$ is the class of all finite soluble groups. By \cite[Remark 4.9]{Zal23}  there exists a second countable space $T$ with the same non-trivial point stabilizers such that $(H,T)$ is a projective group-pile. Since $T$ is second countable, the quotient map $T \to H \bl T$ has  a continuous section (see \cite[Lemma 5.6.7]{RZ} or \cite[1.3.10]{Rib}). By Theorem \ref{piletheo} $(H,T)$ is an $\calS$-projective group-pile such that the quotient map $T \to H \bl T$ has a continuous section. Therefore by Proposition \ref{alternative}  $H$ embeds into a free prosoluble product
$$\sideset{}{^\calS}\coprod_{t \in Im(\sigma)} \ H_t \amalg^{\calS} F$$
where $F$ is a free prosoluble group. Moreover, every $H_t$ is of the form $H \cap G(v)^g$. By Corollary \ref{generalizing pop}, the $H_t$ are either all pro-$p$, or $H$ is Frobenius and all $H_t$ are conjugate as they are Hall subgroups of $H$.
Then by \cite[Theorem 1.4]{Zal23} $H$ embeds into a desired free profinite product.
\end{proof}

\begin{coro}
Let $(\calG,\Gamma)$ be an injective $k$-acylindrical finite graph of profinite groups, $G$ its fundamental group and $S$ its standard profinite graph. If $H$ is a finitely generated  prosoluble subgroup of $G$, then it has only finitely many vertex stabilizers up to conjugation. Moreover, $\sum_v d(H_v)\leq d(H)$, where $v$ runs through representatives of $H$-orbits in $S$. 
\end{coro}

\begin{proof} Follows from Theorem \ref{general} and \cite[Corollary 1.6]{Zal23}.
\end{proof}

\section{Hyperbolic virtually special groups}\label{sec6}

In this section we  apply the previous results to hyperbolic  and relatively hyperbolic virtually compact special groups.

We start with some definitions.

\begin{df}
A group is {\it virtually special} if there is a finite index subgroup isomorphic to the fundamental group of a special cube complex. A group is {\it virtually compact special} if there is a finite index subgroup isomorphic to the fundamental group of a compact special cube complex whose hyperplanes satisfy some extra conditions \cite[Definition 3.2]{HW08}.
\end{df}

\begin{df}
A family $\{H_1,...,H_n\}$ of subgroups of $G$ is {\it malnormal} if, for any $g \in G$, $H_i \cap H_j^g \neq 1$ implies $i = j$ and $g \in H_i$.
\end{df}

\begin{df}
A {\it hierarchy of groups of length $0$} is a single vertex labeled by a group. A {\it hierarchy of groups of length $n$} is a graph of groups $(\calG_n,\Gamma_n)$ together with hierarchies of length $n-1$ on each vertex of $\Gamma_n$ (cf. \cite[Definition 1.3]{WZ17}).
\end{df}

In \cite[Theorem 1.4]{WZ17} it is shown that a hyperbolic group $G$ is virtually compact special if, and only if, $G$ has a subgroup $\Gamma_0$ of finite index with a malnormal hierarchy. It is also shown that, in this case, $\Gamma_0$ has a $1$-acylindrical action on a profinite tree (see \cite[Lemma 7.3]{WZ17}). In the relatively hyperbolic virtually compact special case, $\Gamma_0$, in addition, has a malnormal hierarchy terminating in parabolic subgroups (see \cite[Theorem 3.2]{Zal22new}).

\begin{reptheorem}{virs}
Let $G$ be a torsion-free, hyperbolic, virtually compact special group. Any finitely generated prosoluble subgroup $H$ of the profinite completion of $G$ is projective.
\end{reptheorem}

\begin{proof}
Let $G$ be a finitely generated hyperbolic virtually compact special group and $H$ a finitely generated prosoluble subgroup of $\widehat{G}$. Let $H_0 = \widehat{\Gamma}_0 \cap H$ where $\Gamma_0$ is the subgroup defined above. We use induction on the length of the malnormal hierarchy described above. Since $\widehat{\Gamma_0}$  is a fundamental group of a graph of profinite groups $(\calG,\Gamma)$  whose action on the standard graph is $1$-acylindrical (see \cite[Lemma 7.3]{WZ17}), by Theorem \ref{general}, we have an embedding
$$H_0 \hookrightarrow \coprod_{v\in V(\Gamma)} H_0\cap G(v)^{g_i}, g_i\in \widehat \Gamma_0.$$

Then by Corollary \ref{coro2} either $H_0\leq H_v^g$ or all $H_0\cap G(v)^g$ are pro-$p$. In the first case the result follows immediately from the induction hypothesis. In the second case we  observe first that, by Theorem \ref{general}, $H_0\cap G(v)^g$ are finitely generated which allows us to deduce that $H_0\cap G(v)^g$ are free pro-$p$ (cf. \cite[Theorem F]{WZ17}). It follows that $H_0$ is projective and so $H$ is virtually projective. Since  it is torsion-free, it is projective (see \cite[Proposition 2.5]{Zal04}).
\end{proof}

Theorem \ref{virs} yields a description of prosoluble subgroups of closed hyperbolic $3$-manifolds.

\begin{reptheorem}{projective}
Let $\pi_1(M)$ be the fundamental group of a closed hyperbolic $3$-manifold $M$ and $H$ a finitely generated prosoluble subgroup of the profinite completion of $\pi_1(M)$. Then $H$ is projective.
\end{reptheorem}
  
\begin{proof} By \cite[Theorem 5.4]{AFH20} $\pi_1(M)$ is virtually compact special. Hence the result follows from Theorem \ref{virs}.
\end{proof}

Let $G$ be a residually finite group hyperbolic relative to a malnormal family of parabolic subgroups $\calP = \{P_1,...,P_n\}$.  Any subgroup of a conjugate of $P_i$ will be called {\it parabolic}. A closed subgroup of the profinite completion of a maximal parabolic subgroup of $G$ be called {\it parabolic subgroup} of $\widehat{G}$. 

\begin{prop}\label{rhvcs}
Let $G$ be a relatively hyperbolic virtually compact special group and $H$ a finitely generated prosoluble subgroup of the profinite completion of $G$. Then there is an open subgroup $H_0$ of $H$ such that one of the following statements holds:
\begin{enumerate}[(i)]
\item $H_0$ is isomorphic to a subgroup of a free prosoluble product of pro-$p$ subgroups of parabolic subgroups,

\item $H_0$ is isomorphic to a subgroup of a parabolic subgroup.
\end{enumerate}
\end{prop}

\begin{proof} Let $G$ be a finitely generated relatively hyperbolic virtually compact special group and $H$ a finitely generated prosoluble subgroup of $\widehat{G}$. In \cite[Proof of Theorem 1.3]{Zal22new} we see that there is a subgroup $\Gamma_0$ such that $\widehat{\Gamma_0}$ acts $1$-acylindrically on a profinite tree. Let $H_0 = \widehat{\Gamma}_0 \cap H$. As in the proof of Theorem \ref{virs},  by Theorem \ref{general}, we have an embedding
$$H_0 \hookrightarrow \coprod H_0\cap G(v)^g.$$
We assume w.l.o.g that $H_0$ is torsion-free. Then by Corollary \ref{coro2} either $H_0\leq H_v^g$ or all $H_0\cap G(v)^g$ are pro-$p$. 
In the first case the result follows immediately from the induction hypothesis. 

In the second case we  observe first that by Proposition \ref{finvs} $H_0\cap G(v)^g$ are finitely generated which allows to apply  \cite[ Theorem 1.4]{Zal22new} to deduce that $H_0\cap G(v)^g$ are free pro-$p$ products of parabolic pro-$p$ groups and a free pro-$p$ group. 
\end{proof}

\begin{df}
We say that a relatively hyperbolic group is {\it toral} if it is torsion-free hyperbolic relative to a set $\{A_1,...,A_n\}$ of parabolic subgroups where each $A_i$ is a finitely generated abelian group.
\end{df}

Now we are in the position deduce from Proposition \ref{rhvcs}

\begin{reptheorem}{toral hyperbolic}
Let $G$ be a torsion-free  virtually compact special toral relatively hyperbolic group and $H$ a finitely generated prosoluble subgroup of the profinite completion $\widehat{G}$ of $G$. Then one of the following statements holds:
\begin{enumerate}[(i)]
\item $H$ is virtually a subgroup of a free prosoluble product of  abelian pro-$p$ groups;
\item $H$ is a virtually abelian group.
\end{enumerate}
\end{reptheorem}

Hence we also have as a particular case 

\begin{reptheorem}{sam}
Let $M$ be a  hyperbolic $3$-manifold with cusps and $H$ a finitely generated prosoluble subgroup of the profinite completion of $\pi_1(M)$. Then one of the following statements holds:
\begin{enumerate}[(i)]
\item $H$ is virtually a subgroup of a free prosoluble product of  abelian pro-$p$ groups;

\item $H$ is a virtually abelian.
\end{enumerate}
\end{reptheorem}

\section{The fundamental group of standard compact arithmetic manifolds}\label{sec7}

In this section we give a direct application of Proposition \ref{rhvcs}.

Let $\SO(n,1)$ be the special orthogonal group and $R_k$ the ring of integers of an algebraic number field  $k$. Recall that a hyperbolic $n$-manifold is a manifold $M^n = \H^n/\Gamma$ such that $\Gamma$ is a torsion-free subgroup of $O(n,1)$.

\begin{df}
A {\it standard arithmetic subgroup} $\Gamma$ of $\SO(n,1)$ is a group commensurable with the subgroup $\SO(q,R_k)$ of $R_k$-points of
$$\SO(q) = \{X \in \SL(n+1,\C) : X^tqX = q\}.$$
An arithmetic hyperbolic $n$-manifold $M^n = \H^n/\Gamma$ such that $\Gamma$ is standard arithmetic is a {\it standard arithmetic manifold}.
\end{df}

It is shown in \cite{BHW11} that standard uniform arithmetic subgroups of $\SO(n,1)$ are virtually compact special and in \cite[Lemma 6.3 and Theorem 7.4]{PS24} this was extended to standard non-uniform arithmetic subgroups. Also, it is well-known that the fundamental group of these manifolds is hyperbolic relative to virtually abelian subgroups. Thus, Theorems \ref{virs} and  \ref{toral hyperbolic} can be applied to standard arithmetic  manifolds.  

\begin{theo}
Let $M$ be a standard   arithmetic manifold and $H$ a finitely generated prosoluble subgroup of the profinite completion of $\pi_1(M)$. One of the following holds:

\begin{enumerate}[(i)]
\item $H$ is  a subgroup of a free prosoluble product of  abelian pro-$p$ groups;

\item $H$ is  abelian.
\end{enumerate}
Moreover, if $\pi_1(M)$ is uniform, then $H$ is projective.
\end{theo}

\section{The fundamental group of compact $3$-manifolds}\label{sec8}

To prove Theorem \ref{ori} we will use the pro-$p$ version of it, i.e., the classification of pro-$p$ subgroups of the profinite completion of a 3-manifold group established by Henry Wilton and the second author in \cite[Theorem 1.3]{WZ18}. 

\begin{theo}{\cite[Theorem 1.3]{WZ18}}\label{henry}
A finitely generated pro-$p$ subgroup of the profinite completion of the fundamental group of a compact, orientable $3$-manifold $M$ is a free pro-$p$ product of  pro-$p$ groups from the following list of isomorphism types:
\begin{enumerate}[(i)]
\item For $p > 3$: $C_p$; $\Z_p$; $\Z_p \times \Z_p$; the pro-$p$ completion of $(\Z \times \Z) \rtimes \Z$ and the pro-$p$ completion of a residually-$p$ fundamental group of a non-compact Seifert fibred manifold with hyperbolic base of orbifold;

\item For $p = 3$: in addition to the list of (i) we have a torsion-free extension of $\Z_3 \times \Z_3 \times \Z_3$ by $C_3$;

\item For $p = 2$: in addition to the list of (i) we have $C_{2^m}$; $D_{2^k}$; $Q_{2^n}$; $\Z_2 \rtimes C_2$; the torsion-free extensions of $\Z_2 \times \Z_2 \times \Z_2$ by one of $C_2$, $C_4$, $C_8$, $D_2$, $D_4$, $D_8$, $Q_{16}$; the pro-$2$ extension of the Klein-bottle group $\Z \rtimes \Z$; the pro-$2$ completion of all torsion-free extensions of a soluble group $(\Z \rtimes \Z) \rtimes \Z$ with a group of order at most $2$.
\end{enumerate}
\end{theo}

Now we are ready to prove the theorem that classifies prosoluble subgroups of the profinite completion of 3-manifold groups.

\begin{reptheorem}{ori}
Let $M$ be a compact orientable $3$-manifold. If $H$ is a finitely generated prosoluble subgroup of $\widehat{\pi_1(M)}$, then one of the following statements holds:
\begin{enumerate}[(i)]
\item $H$ is isomorphic to a subgroup of a free prosoluble product of pro-$p$ groups from the following list of isomorphism types:
\begin{enumerate}[(1)]
\item For $p > 3$: $C_p$; $\Z_p$; $\Z_p \times \Z_p$; the pro-$p$ completion of $(\Z \times \Z) \rtimes \Z$ and the pro-$p$ completion of a residually-$p$ fundamental group of a non-compact Seifert fibred manifold with hyperbolic base of orbifold;

\item For $p = 3$: in addition to the list of (1) we have a torsion-free extension of $\Z_3 \times \Z_3 \times \Z_3$ by $C_3$;

\item For $p = 2$: in addition to the list of (1) we have $C_{2^m}$; $D_{2^k}$; $Q_{2^n}$; $\Z_2 \rtimes C_2$; the torsion-free extensions of $\Z_2 \times \Z_2 \times \Z_2$ by one of $C_2$, $C_4$, $C_8$, $D_2$, $D_4$, $D_8$, $Q_{16}$; the pro-$2$ completion of the Klein-bottle group $\Z \rtimes \Z$; the pro-$2$ completion of all torsion-free extensions of a soluble group $(\Z \rtimes \Z) \rtimes \Z$ with a group of order at most $2$;
\end{enumerate}

\item $H \simeq \widehat{\Z}_{\pi} \rtimes C$ is a profinite Frobenius group where $C$ is a finite cyclic group and $\pi$ is set of primes;

\item $H$ is a subgroup of a central extension of either $D_{2n}$ or tetrahedral group $T$ or octahedral group $O$ by a cyclic group of even order;

\item $H$ is a subgroup of the profinite completion of a $3$-dimensional Bieberbach group $B$ (i.e, $B$ is torsion-free virtually $\Z^3$);

\item $H$ is a subgroup of the profinite completion of a group extension of $\Z^2 \rtimes_A \Z$ by $C$ where $C$ is trivial or $C_2$ and $A \in \GL_2(\Z)$ (and $A$ is an Anosov matrix if $C$ is trivial), or a central extension of $\Z$ by $\Z^2$;

\item $H$ is an extension of a torsion-free procyclic group $\widehat{\Z}_{\sigma}$ by a subgroup $H_0$ of a free prosoluble product of finitely many  finite cyclic $p$-groups with $H_0$ acting either trivially on $\widehat{\Z}_{\sigma}$  or by inversion.
\end{enumerate}
\end{reptheorem}

\subsection{Reduction to the irreducible case}
 
Let $M$ be a compact, orientable, non-geometric $3$-manifold. Since every compact manifold is a retract of a closed manifold, it follows easily that we may reduce to the case in which $M$ is closed. 

If $M$ is reducible, by the Prime Decomposition Theorem, we can write
$$\pi_1(M) = \pi_1(M_1) \ast \cdots \ast \pi_1(M_n)$$
where each $M_i$ is a prime $3$-manifold. A prime $3$-manifold is either irreducible or $S^2 \times S^1$ or the non-orientable $S^2$ bundle over $S^1$. In the last two cases, the fundamental group is isomorphic to $\Z$ (see \cite{Hem04} and \cite[Page 5]{GSW21}). 

Thus applying \cite[Theorem 1.4]{Zal23} we obtain

\begin{theo}
Let $H$ be  a finitely generated prosoluble subgroup of the profinite completion $G=\widehat{\pi_1(M)}$ of a reducible $3$-manifold having a non-trivial free product decomposition $$\pi_1(M) = \pi_1(M_1) \ast \cdots \ast \pi_1(M_n)\ast F$$
where each $M_i$ is a irreducible $3$-manifold and $F$ is a free group of finite rank. Then one of the following holds:

\begin{enumerate}[(i)]
\item all intersections $H\cap \widehat{\pi_1(M_i)}^g$, $g\in G$ are pro-$p$ for some fixed prime $p$ and $\sum_v d(H\cap \widehat{\pi_1(M_i)}^g)\leq d(H)$;

\item  $H \simeq \widehat{\Z}_\pi\rtimes C$ is a  profinite Frobenius group;
  
\item $H\leq  \widehat{\pi_1(M_i)}^g$ for some $g\in G$.
\end{enumerate}
\end{theo}

In case $(i)$ we apply Theorem \ref{henry} to $H\cap \widehat{\pi_1(M_i)}^g$ to deduce the structure of each of these intersections and obtain $(i)$ of Theorem \ref{ori}. Case $(ii)$ of the theorem above corresponds to $(ii)$ of Theorem \ref{ori}. 

Thus we only need to consider $(iii)$ and in this case we may assume, w.l.o.g, that $H\leq  \widehat{\pi_1(M_i)}$, i.e., we may assume that $M$ is irreducible.

\subsection{Reduction to the geometric cases}

Now we can use the following:

\begin{prop}[\cite{WZ18}, Proposition 4.1]
Let $M$ be a closed irreducible $3$-manifold. Then the action of $\widehat{\pi_1(M)}$ on the standard profinite tree $S$ associated to the JSJ-decomposition of $M$ is $4$-acylindrical.
\end{prop}

Then we can apply Theorem \ref{general} to deduce:

\begin{theo}\label{irr}
Let $M$ be a closed irreducible $3$-manifold. If $H$ is a finitely generated prosoluble subgroup of $\widehat{\pi_1(M)}$, then
$H$ embeds into a free profinite product $$\coprod_{v \in V} H_v$$ of $H$-stabilizers of vertices in $S$ (the standard profinite tree associated to $\widehat{\pi_1(M)}$) and one of the following statements holds:
\begin{enumerate}[(i)]
\item $H$ is isomorphic to a subgroup of a free prosoluble product of finitely many finitely generated pro-$p$ groups $H \cap u H_v u^{-1}$ for every $v \in V$ and $u \in \coprod_{v \in V} H_v$;

\item $H$ is a profinite Frobenius group $\widehat{\Z}_\pi\rtimes C$;

\item $H$ is a subgroup of $H_v$ for some $v \in V$.
\end{enumerate}
\end{theo}

If $H$ satisfies $(i)$, then we can apply Theorem \ref{henry} again and deduce $(i)$ or $(ii)$ of Theorem \ref{ori}. The case $(ii)$ corresponds to Theorem \ref{ori} $(ii)$. Therefore we may assume that $H$ satisfies item $(iii)$ and since $H$ in this case is conjugate into a vertex of $\widehat{\pi_1(M)}$ we may assume that $H$ is contained in the profinite completion of a geometric manifold. 

Thus, to complete the proof of Theorem \ref{ori} we need to study the profinite completion of $\pi_1(M)$ when $M$ is geometric. 

\subsection{Geometric cases}

The Geometrization conjecture, proved by G. Perelman on his approach to prove the Poincaré conjecture (see \cite{Per02}, \cite{Per03}, \cite{Per032}), says that a $3$-manifold is geometrically modeled by one of the eight Thurston's geometries
\begin{multicols}{4}
\begin{enumerate}[(i)]
\item $S^3$,

\item $S^2 \times \R$,

\item $\R^3$,

\item $Nil$,

\item $Sol$

\item $\H^3$,

\item $\H^2 \times \R$,

\item $\widetilde{\SL_2(\R)}$.
\end{enumerate}
\end{multicols}

The fundamental group of the first five are, respectively, of the following types:

\begin{enumerate}[(i)]
\item a finite cyclic group or a central extension of $D_{2n}$, $T$, $O$ or $I$ by a cyclic group of even order. We are denoting by $T$ the tetrahedral group of $12$ rotational symmetries of a tetrahedron, by $O$ the octahedral group of $24$ rotational symmetries of a cube or of an octahedron, by $I$ the icosahedral group of $60$ rotational symmetries of a dodecahedron or an icosahedron;

\item isomorphic to $\Z$ or the infinite dihedral group;

\item virtually abelian;

\item virtually nilpotent;

\item soluble.
\end{enumerate}

The fundamental groups in the case (i) are soluble. except  the case of  an extension of $I$ by a cyclic group of even order. The icosahedral group $I$ is isomorphic to $A_5$, which is non-soluble. Recall that $A_5$ has no proper subgroups of order greater than $12$ so the possible subgroups are $C_2, C_3$, $C_5$, $D_3, D_5, A_4$ and they are covered in the previous cases. Thus, to analyze the soluble subgroups when $\Gamma$ is a central extension of $A_5$ by $C_{2n}$ we can reduce it to the other soluble cases. 

If we consider $H$ as a prosoluble subgroup of the profinite completion of  groups (i)-(v), then $H$ can be any of its subgroups. We summarize the existing classification as:

\begin{enumerate}[(i)]
\item $H$ is a finite cyclic group or $H$ is a subgroup of a central extension of a dihedral, tetrahedral, octahedral or icosahedral group by a cyclic group of even order;

\item $H$ is a subgroup of the profinite completion of the infinite dihedral group;

\item $H$ is a subgroup of the profinite completion of a $3$-dimensional Bieberbach group (i.e, virtually $\Z^3$);

\item $H$ is a subgroup of the profinite completion of a group extension of $\Z^2 \rtimes_A \Z$ by $C$ where $C$ is trivial or isomorphic to $C_2$ and $A \in \GL_2(\Z)$ or a central extension of $\Z$ by $\Z^2$;
\end{enumerate}

The reader can check \cite{AFH20}, \cite[Section 4]{Sco83}, \cite[Section 1.5, Chapter 2, Chapter 3]{Gen21} and \cite{GM23} for details on the fundamental groups of the first five items.

The item $(vi)$ (hyperbolic case) is the subject of Theorem \ref{projective}. Since a finitely generated projective profinite group is a subgroup of a non dihedral free profinite product (cf. Proposition \ref{new}) it is a particular case of Theorem \ref{ori}, item $(i)$. So, it remains to deal with cases $(vii)$ and $(viii)$.

\subsection{Seifert manifolds}

In order to conclude the analysis of the possible irreducible compact $3$-manifolds we now deal with the Seifert manifolds. We have the following exact sequence:
$$1 \to \Z \to \pi_1(M) \to \pi_1(O) \to 1$$
where $\pi_1(O)$ is a Fuchsian group acting on $\Z$. Thus we describe first prosoluble subgroups of a profinite completion of a  Fuchsian group.

\begin{theo}\label{fuchsian}
Let $H$ be a prosoluble subgroup of a profinite completion of a Fuchsian group $\widehat{\pi_1(O)}$. Then one of the following holds:

\begin{enumerate}[(i)]
\item $H$ is a subgroup of a free prosoluble product of finitely many  finite cyclic $p$-groups;

\item $H\simeq \widehat{\Z}_\pi\rtimes C$ is a profinite Frobenius group;

\item $H$ is a finite cyclic group.
\end{enumerate}
\end{theo} 

\begin{proof}
Passing to an open subgroup containing $H$ if necessary (see \cite{HKS71}), we can replace $O$ by a non-trivial finite-sheeted cover of degree $n > 2$ whose profinite completion contains $H$, hence, we can assume that $\pi_1(O)$ is not a triangle group. In this case we have
$$\widehat{\pi_1(O)} = \widehat{K} \amalg_{\widehat{C}} \widehat{F}$$
where $K$ is a free product of cyclic groups, $F$ a free profinite group of finite rank and $C$ an infinite cyclic group (see \cite{LMR96}). 

Since $\widehat{C}$ is malnormal in $\widehat{\pi_1(O)}$, by \cite[Lemma 7.3]{WZ17}, $\widehat{\pi_1(O)}$ is a $1$-acylindrical finite graph of profinite groups. Applying Theorem \ref{general} we get one of the following cases:

\begin{itemize}
\item[(a)] $H$ is a subgroup of the free prosoluble product of intersections $H\cap \widehat{F}^h$, $h \in \widehat{\pi_1(O)}$ and $H\cap \widehat{K}^g$,  $g \in \widehat{\pi_1(O)}$, and these intersections are pro-$p$.

\item[(b)] $H$ is a profinite Frobenius group.

\item[(c)] $H$ is contained in either $\widehat{F}$ or $\widehat{K}$ up to conjugation.
\end{itemize}

Suppose that (a) holds. Note that $H\cap \widehat{F}^g$ is free pro-$p$ and $H\cap \widehat{K}^g$ is free pro-$p$ product of cyclic pro-$p$ groups by the pro-$p$ version of the Kurosh Subgroup Theorem (cf. \cite[Theorem 9.6.2]{Rib}). Then we have $(i)$. If  $H$ is Frobenius, we have item $(ii)$.

Suppose that (c) holds. If $H$ is the conjugate of $\widehat{F}$ then it is projective and so $(i)$ holds. Otherwise we may assume that $H\leq \widehat{K}$. Then applying Theorem \ref{general} again we deduce the statement of the theorem.
\end{proof}

\begin{theo}
Let $G$ be the fundamental group of a Seifert manifold with either geometry $\H^2 \times \R$ or $\widetilde{\SL_2(\R)}$ and $H$ a finitely generated prosoluble subgroup of $\widehat{G}$. Then  $H$ is an extension of a torsion-free procyclic group $\widehat{\Z}_{\sigma}$ by a subgroup of a  free prosoluble product of finitely many finite cyclic $p$-groups.
\end{theo}

\begin{proof}
We have the following exact sequence:
$$1 \to \Z \to \pi_1(M) \to \pi_1(O) \to 1$$
where $\pi_1(O)$ is a Fuchsian group acting on $\Z$ either trivially or by inversion.  
Then we have an action of $\widehat{\pi_1(O)}$ and hence  of $H_0=H\widehat{\Z}/\widehat{\Z}$ on $\widehat{\Z}$  and  this action is trivial or by inversion.

Thus we have an exact sequence
$$1 \to \widehat{\Z}_{\sigma} \to H \to H_0 \to 1,$$ where $\widehat{\Z}_\sigma=H\cap \widehat \Z$ and $H_0$ acts on $\widehat{\Z}_\sigma$ either trivially or by inversion.

By Theorem \ref{fuchsian} we have the following 3 possibilities for $H_0$.
\begin{enumerate}
\item[(a)] $H_0$ is a subgroup of a free prosoluble product of finitely many  finite cyclic $p$-groups;

\item[(b)] $H_0 \simeq \widehat{\Z}_{\pi} \rtimes C$ is a profinite Frobenius group;

\item[(c)] $H_0$ is a finite cyclic group.
\end{enumerate}
and we analyze each of them in turn.
\begin{enumerate}[(a)]
\item[(a)] In this case $H$ is an extension of $\widehat{\Z}_{\sigma}$ by $H_0$ with the  either trivial action  or by inversion.

\item[(b)] Suppose that
$$1 \to \widehat{\Z}_{\sigma} \to H \to \widehat{\Z}_{\pi} \rtimes C_n \to 1$$
is an exact sequence with $\widehat{\Z}_{\pi} \rtimes C_n$ a profinite Frobenius group with $C_n$ a finite cyclic group of order $n$. As the lift of every element of finite order of $\pi_1(O)$ to $\pi_1(M)$ must centralize $\Z$, the lift of every element of finite order of $H_0$ must centralize $\widehat{\Z}_{\sigma}$. But $H_0$ is generated by torsion elements, so the extension is central in this case. 

Let $\rho$ be the set of all primes dividing $n$. Then the preimage of $C_n$ in $H$ is $ \widehat{\Z}_{\rho'} \times \widehat{\Z}_{\rho}$ where $\rho'$ is the complement of $\rho$ in $\sigma$ (it is isomorphic but not equal to $\widehat{\Z}_{\sigma}$). Thus we can write $H$ as ($\widehat{\Z}_{\rho'} \times \widehat{\Z}_\pi)\rtimes \widehat{\Z}_{\rho}\cong \widehat{\Z}_{\rho'} \rtimes (\widehat{\Z}_\pi\rtimes \widehat{\Z}_{\rho})$). As $\pi\cap \rho=\emptyset$, $ \widehat{\Z}_\pi \rtimes \widehat{\Z}_{\rho}$  is projective and hence a subgroup of a free profinite product of cyclic $p$-groups (since a finitely generated projective group is a subgroup of a free group of finite rank).

\item[(c)] Suppose that
$$1 \to \widehat{\Z}_{\sigma} \to H \to C_n \to 1$$
is an exact sequence with $C_n$ finite cyclic of order $n$. As the lift of every element of finite order of $\pi_1(O)$ to $\pi_1(M)$ must centralize $\Z$, the lift of every element of finite order of $H_0$ must centralize $\widehat{\Z}_{\sigma}$. Hence the extension  is central and so $H$ is procyclic torsion-free.
\end{enumerate}
\end{proof}




\begin{thebibliography}{BMRS20}

\bibitem[Ago13]{Ago13}
I.~Agol.
\newblock The virtual Haken conjecture (with an appendix by Ian Agol, Daniel
  Groves and Jason Manning).
\newblock {\em Documenta Mathematica}, 18:1045--1087, 2013.
\newblock \href {https://doi.org/10.4171/dm/421} {\path{doi:10.4171/dm/421}}.

\bibitem[AZ22]{AZ}
M.~P.~S. Aguiar and P.~A. Zalesskii.
\newblock The profinite completion of the fundamental group of infinite graphs
  of groups.
\newblock {\em Israel Journal of Mathematics}, 250(1):429--462, 2022.
\newblock \href {https://doi.org/10.1007/s11856-022-2342-2}
  {\path{doi:10.1007/s11856-022-2342-2}}.

\bibitem[AFW15]{AFH20}
M.~Aschenbrenner, S.~Friedl, and H.~Wilton.
\newblock {\em $3$-manifold groups}, volume~20.
\newblock European Mathematical Society Z{\"u}rich, 2015.
\newblock URL: \url{https://ems.press/books/elm/217}, \href
  {https://doi.org/10.4171/154} {\path{doi:10.4171/154}}.

\bibitem[BHW11]{BHW11}
N.~Bergeron, F.~Haglund, and D.~T. Wise.
\newblock Hyperplane sections in arithmetic hyperbolic manifolds.
\newblock {\em Journal of the London Mathematical Society}, 83(2):431--448,
  2011.
\newblock \href {https://doi.org/10.1112/jlms/jdq082}
  {\path{doi:10.1112/jlms/jdq082}}.

\bibitem[BMRS20]{BMRS20}
M.~A. Bridson, D.~B. McReynolds, A.~W. Reid, and R.~Spitler.
\newblock Absolute profinite rigidity and hyperbolic geometry.
\newblock {\em Annals of Mathematics}, 192(3):679--719, 2020.
\newblock \href {https://doi.org/10.4007/annals.2020.192.3.1.code}
  {\path{doi:10.4007/annals.2020.192.3.1.code}}.

\bibitem[BRW17]{BRW17}
M.~R. Bridson, A.~W. Reid, and H.~Wilton.
\newblock Profinite rigidity and surface bundles over the circle.
\newblock {\em Bulletin of the London Mathematical Society}, 49(5):831--841,
  2017.
\newblock \href {https://doi.org/10.1112/blms.12076}
  {\path{doi:10.1112/blms.12076}}.

  \bibitem[CZ23]{CZ}
I.~Castellano and P.~Zalesskii.
\newblock A pro-$p$ version of Sela's accessibility and Poincaré duality
  pro-$p$ groups.
\newblock {\em Groups Geom. Dyn.}, 18:1349–1368, 2024.
\newblock URL: \url{https://doi.org/10.4171/ggd/769}, \href
  {https://doi.org/10.4171/GGD/769} {\path{doi:10.4171/GGD/769}}.

\bibitem[CZ22]{ChZ22}
Z.~Chatzidakis and P.~Zalesskii.
\newblock Pro-$p$ groups acting on trees with finitely many maximal vertex
  stabilizers up to conjugation.
\newblock {\em Israel Journal of Mathematics}, 247(2):593--634, 2022.
\newblock \href {https://doi.org/10.1007/s11856-022-2287-5}
  {\path{doi:10.1007/s11856-022-2287-5}}.

\bibitem[CW22]{CW22}
T.~Cheetham-West.
\newblock Absolute profinite rigidity of some closed fibered hyperbolic
  3-manifolds. 
  \newblock {\em Mathematical Research Letters},  31:615--638, 2024.
\newblock {\em arXiv preprint arXiv:2205.08693}, 2022.

\bibitem[Coh06]{Coh06}
D.~E. Cohen.
\newblock Group with free subgroups of finite index.
\newblock In {\em Conference on Group Theory: University of Wisconsin-Parkside
  1972}, pages 26--44. Springer, 2006.
\newblock \href {https://doi.org/10.1007/BFb0058927}
  {\path{doi:10.1007/BFb0058927}}.

\bibitem[Ein19]{Ein19}
E.~Einstein.
\newblock Hierarchies for relatively hyperbolic virtually special groups.
\newblock {\em arXiv preprint arXiv:1903.12284}, 2019.
\newblock \href {https://doi.org/10.48550/arXiv.1903.12284}
  {\path{doi:10.48550/arXiv.1903.12284}}.

\bibitem[GSW21]{GSW21}
D.~Gon{\c{c}}alves, P.~Sankaran, and P.~Wong.
\newblock Twisted conjugacy in fundamental groups of geometric $3$-manifolds.
\newblock {\em Topology and its Applications}, 293:107568, 2021.
\newblock \href {https://doi.org/10.1016/j.topol.2020.107568}
  {\path{doi:10.1016/j.topol.2020.107568}}.

\bibitem[GM23]{GM23}
D.~L. Gon{\c{c}}alves and S.~T. Martins.
\newblock The groups Aut and Out of the fundamental group of a closed Sol
  $3$-manifold.
\newblock {\em Journal of Algebra and Its Applications}, 22, 2023.
\newblock \href {https://doi.org/10.1142/S0219498823501980}
  {\path{doi:10.1142/S0219498823501980}}.

\bibitem[HW08]{HW08}
F.~Haglund and D.~T. Wise.
\newblock Special cube complexes.
\newblock {\em Geometric and Functional Analysis}, 17:1551--1620, 2008.
\newblock \href {https://doi.org/10.1007/s00039-007-0629-4}
  {\path{doi:10.1007/s00039-007-0629-4}}.

\bibitem[HW12]{HW12}
F.~Haglund and D.~T. Wise.
\newblock A combination theorem for special cube complexes.
\newblock {\em Annals of Mathematics}, 176:1427--1482, 2012.
\newblock URL: \url{http://doi.org/10.4007/annals.2012.176.3.2}, \href
  {https://doi.org/10.4007/annals.2012.176.3.2}
  {\path{doi:10.4007/annals.2012.176.3.2}}.

\bibitem[Har87]{Har87}
D.~Haran.
\newblock On closed subgroups of free products of profinite groups.
\newblock {\em Proceedings of the London Mathematical Society}, 3(2):266--298,
  1987.
\newblock \href {https://doi.org/10.1093/plms/s3-55_2.266}
  {\path{doi:10.1093/plms/s3-55_2.266}}.

\bibitem[HZ23]{HZ}
D.~Haran and P.~A. Zalesskii.
\newblock Relatively projective pro-$p$ groups.
\newblock {\em Israel Journal of Mathematics}, 257(2):313--352, 2023.
\newblock \href {https://doi.org/10.1007/s11856-023-2544-2}
  {\path{doi:10.1007/s11856-023-2544-2}}.

\bibitem[Hem04]{Hem04}
J.~Hempel.
\newblock {\em $3$-Manifolds}, volume 349.
\newblock American Mathematical Soc., 2004.

\bibitem[HKS71]{HKS71}
A.~H.~M. Hoare, A.~Karrass, and D.~Solitar.
\newblock Subgroups of finite index of Fuchsian groups.
\newblock {\em Mathematische Zeitschrift}, 120(4):289--298, 1971.
\newblock \href {https://doi.org/10.1007/BF01109993}
  {\path{doi:10.1007/BF01109993}}.

\bibitem[Jar94]{Jar94}
M.~Jarden.
\newblock Prosolvable subgroups of free products of profinite groups.
\newblock {\em Communications in Algebra}, 22(4):1467--1494, 1994.
\newblock \href {https://doi.org/10.1080/00927879408824917}
  {\path{doi:10.1080/00927879408824917}}.

\bibitem[KM12]{JV12}
J.~Kahn and V.~Markovic.
\newblock Immersing almost geodesic surfaces in a closed hyperbolic three
  manifold.
\newblock {\em Annals of Mathematics}, pages 1127--1190, 2012.
\newblock \href {https://doi.org/10.4007/annals.2012.175.3.4}
  {\path{doi:10.4007/annals.2012.175.3.4}}.

\bibitem[KPS73]{KPS73}
A.~Karrass, A.~Pietrowski, and D.~Solitar.
\newblock Finite and infinite cyclic extensions of free groups.
\newblock {\em Journal of the Australian Mathematical Society}, 16(4):458--466,
  1973.
\newblock \href {https://doi.org/10.1017/S1446788700015445}
  {\path{doi:10.1017/S1446788700015445}}.

\bibitem[LMR96]{LMR96}
D.~D. Long, C.~Maclachlan, and A.~W. Reid.
\newblock Splitting groups of signature $(1;n)$.
\newblock {\em Journal of Algebra}, 185(2):329--341, 1996.
\newblock \href {https://doi.org/10.1006/jabr.1996.0328}
  {\path{doi:10.1006/jabr.1996.0328}}.

\bibitem[Ner21]{Gen21}
G.~J. Nery.
\newblock {\em G{\^e}nero profinito de grupos de variedades planas e
  sol{\'u}veis}.
\newblock PhD thesis, Universidade de Bras{\'i}lia, 2021.

\bibitem[Per02]{Per02}
G.~Perelman.
\newblock The entropy formula for the Ricci flow and its geometric
  applications.
\newblock {\em arXiv preprint math/0211159}, 2002.
\newblock \href {https://doi.org/10.48550/arXiv.math/0211159}
  {\path{doi:10.48550/arXiv.math/0211159}}.

\bibitem[Per03a]{Per032}
G.~Perelman.
\newblock Finite extinction time for the solutions to the Ricci flow on certain
  three-manifolds.
\newblock {\em arXiv preprint math/0307245}, 2003.
\newblock \href {https://doi.org/10.48550/arXiv.math/0307245}
  {\path{doi:10.48550/arXiv.math/0307245}}.

\bibitem[Per03b]{Per03}
G.~Perelman.
\newblock Ricci flow with surgery on three-manifolds.
\newblock {\em arXiv preprint math/0303109}, 2003.
\newblock \href {https://doi.org/10.48550/arXiv.math/0303109}
  {\path{doi:10.48550/arXiv.math/0303109}}.

\bibitem[PS24]{PS24}
N.~Petrosyan and B.~Sun.
\newblock $L^2$-betti numbers of Dehn fillings.
\newblock {\em arXiv preprint arXiv:2412.16090}, 2024.
\newblock \href {https://doi.org/10.48550/arXiv.2412.16090}
  {\path{doi:10.48550/arXiv.2412.16090}}.

\bibitem[Rib17]{Rib}
L.~Ribes.
\newblock {\em Profinite graphs and groups}, volume~66.
\newblock Springer, 2017.
\newblock \href {https://doi.org/10.1007/978-3-319-61199-0}
  {\path{doi:10.1007/978-3-319-61199-0}}.

\bibitem[RZ10]{RZ}
L.~Ribes and P.~Zalesskii.
\newblock {\em Profinite groups}.
\newblock Springer, 2010.
\newblock \href {https://doi.org/10.1007/978-3-642-01642-4}
  {\path{doi:10.1007/978-3-642-01642-4}}.

\bibitem[Sco74]{Sco74}
G.~P. Scott.
\newblock An embedding theorem for groups with a free subgroup of finite index.
\newblock {\em Bulletin of the London Mathematical Society}, 6(3):304--306,
  1974.
\newblock \href {https://doi.org/10.1112/blms/6.3.304}
  {\path{doi:10.1112/blms/6.3.304}}.

\bibitem[Sco83]{Sco83}
P.~Scott.
\newblock The geometries of $3$-manifolds.
\newblock {\em Bulletin of the London Mathematical Society}, 1983.
\newblock \href {https://doi.org/10.1112/blms/15.5.401}
  {\path{doi:10.1112/blms/15.5.401}}.

\bibitem[Sel97]{Sel97}
Z.~Sela.
\newblock Acylindrical accessibility for groups.
\newblock {\em Inventiones mathematicae}, 129(3):527--565, 1997.
\newblock \href {https://doi.org/10.1007/s002220050172}
  {\path{doi:10.1007/s002220050172}}.

\bibitem[Ser02]{Ser}
J-P. Serre.
\newblock {\em Trees}.
\newblock Springer Science \& Business Media, 2002.
\newblock \href {https://doi.org/10.1007/978-3-642-61856-7}
  {\path{doi:10.1007/978-3-642-61856-7}}.

\bibitem[Wil18]{Wil18}
G.~Wilkes.
\newblock Profinite rigidity of graph manifolds and JSJ decompositions of
  $3$-manifolds.
\newblock {\em Journal of Algebra}, 502:538--587, 2018.
\newblock \href {https://doi.org/10.1016/j.jalgebra.2017.12.039}
  {\path{doi:10.1016/j.jalgebra.2017.12.039}}.

\bibitem[WZ17]{WZ17}
H.~Wilton and P.~Zalesskii.
\newblock Distinguishing geometries using finite quotients.
\newblock {\em Geometry $\&$ Topology}, 21(1):345--384, 2017.
\newblock \href {https://doi.org/10.2140/gt.2017.21.345}
  {\path{doi:10.2140/gt.2017.21.345}}.

\bibitem[WZ18]{WZ18}
H.~Wilton and P.~Zalesskii.
\newblock Pro-$p$ subgroups of profinite completions of $3$-manifold groups.
\newblock {\em Journal of the London Mathematical Society}, 96(2):293--308,
  2018.
\newblock \href {https://doi.org/10.1112/jlms.12067}
  {\path{doi:10.1112/jlms.12067}}.

\bibitem[Wis21]{Wis11}
D.~T. Wise.
\newblock The structure of groups with a quasiconvex hierarchy.
\newblock 366, 2021.
\newblock \href {https://doi.org/10.2307/j.ctv1574pr6}
  {\path{doi:10.2307/j.ctv1574pr6}}.

\bibitem[Zal04]{Zal04}
P.~A. Zalesskii.
\newblock On virtually projective groups.
\newblock {\em Journal für die reine und angewandte Mathematik},
  2004(572):97--110, 2004.
\newblock \href {https://doi.org/10.1515/crll.2004.055}
  {\path{doi:10.1515/crll.2004.055}}.

\bibitem[Zal23a]{Zal22new}
P.~Zalesskii.
\newblock The profinite completion of relatively hyperbolic virtually special
  groups.
\newblock {\em Israel Journal of Mathematics}, pages 1--14, 2023.
\newblock \href {https://doi.org/10.1007/s11856-023-2584-7}
  {\path{doi:10.1007/s11856-023-2584-7}}.

\bibitem[Zal23b]{Zal22}
P.~Zalesskii.
\newblock Relatively projective profinite groups.
\newblock {\em European Journal of Mathematics}, 9, 2023.
\newblock \href {https://doi.org/10.1007/s40879-023-00705-1}
  {\path{doi:10.1007/s40879-023-00705-1}}.

\bibitem[Zal24]{Zal23}
P. Zalesskii.
\newblock Prosoluble subgroups of free profinite products.
\newblock {\em Forum of Mathematics, Sigma (to appear)}, 2024.
\newblock \href {https://doi.org/10.1017/fms.2024.92}
  {\path{doi:10.1017/fms.2024.92}}.

\bibitem[ZM89a]{ZM-89}
P.~A. Zalesskii and O.~V. Mel'nikov.
\newblock Fundamental groups of graphs of profinite groups.
\newblock {\em Algebra i Analiz}, 1(4):117--135, 1989.

\bibitem[ZM89b]{ZM88}
P.~A. Zalesskii and O.~V. Mel'nikov.
\newblock Subgroups of profinite groups acting on trees.
\newblock {\em Mathematics of the USSR-Sbornik}, 63(2):405, 1989.
\newblock \href {https://doi.org/10.1070/SM1989v063n02ABEH003282}
  {\path{doi:10.1070/SM1989v063n02ABEH003282}}.

\end{thebibliography}
\end{document}